\documentclass[11pt,reqno,tbtags]{amsart}

\usepackage{amsthm,amssymb,amsfonts,amsmath}
\usepackage{amscd} 
\usepackage[numbers, sort&compress]{natbib} 
\usepackage{subfigure, graphicx}
\usepackage{vmargin}
\usepackage{setspace}
\usepackage{paralist}
\usepackage[pagebackref=true]{hyperref}
\usepackage{color}
\usepackage[normalem]{ulem}
\usepackage{stmaryrd} 
\usepackage[refpage,noprefix,intoc]{nomencl} 
\usepackage{tcolorbox}

\providecommand{\N}{}

\renewcommand{\N}{{\mathbb N}}

\newcommand{\E}[1]{{\mathbf E}\left[#1\right]}				
\newcommand{\e}{{\mathbf E}}

\newcommand{\p}[1]{{\mathbf P}\left\{#1\right\}}

\newcommand{\I}[1]{{\mathbf 1}_{[#1]}}
\newcommand{\set}[1]{\left\{ #1 \right\}}
\newcommand{\Cprob}[2]{\mathbf{P}\set{\left. #1 \; \right| \; #2}} 
\newcommand{\probC}[2]{\mathbf{P}\set{#1 \; \left|  \; #2 \right. }}
\newcommand{\Cexp}[2]{\mathbf{E}\left[\left. #1 \; \right| \; #2\right]}

\newcommand\cF{\mathcal F}

\newcommand\cL{{\mathcal L}}

\newcommand\cN{\mathcal N}

\newcommand\cP{\mathcal P}

\newcommand{\pran}[1]{\left(#1\right)}
\providecommand{\eps}{}
\renewcommand{\eps}{\varepsilon}
\providecommand{\ora}[1]{}
\renewcommand{\ora}[1]{\overrightarrow{#1}}
\DeclareRobustCommand{\SkipTocEntry}[5]{} 

\reversemarginpar
\newtheorem{thm}{Theorem}
\newtheorem{lem}[thm]{Lemma}
\newtheorem{prop}[thm]{Proposition}

\newtheorem{cor}[thm]{Corollary}

\numberwithin{equation}{section}
\numberwithin{thm}{section}

\makenomenclature										
\newcommand{\geom}{\operatorname{Geom}}

\graphicspath{ {./images/} }
\usepackage{wrapfig}

\begin{document}

\title{Random friend trees} 
\author[Addario-Berry]{Louigi Addario-Berry}
\address{Department of Mathematics and Statistics, McGill University, 
		Montr\'eal, Qu\'ebec,  Canada}
\email{louigi.addario@mcgill.ca}
\author[Briend]{Simon Briend}
\address{Université Paris-Saclay, CNRS \\ Laboratoire de Mathématiques d'Orsay \\ 91404, Orsay, France}
\email{simon.briend@universite-paris-saclay.fr}
\author[Devroye]{Luc Devroye}
\address{School of Computer Science, McGill University, 
		Montr\'eal, Qu\'ebec,  Canada}
\email{lucdevroye@gmail.com}
\author[Donderwinkel]{Serte Donderwinkel} 
\address{Bernoulli Institute and CogniGron (Groningen Cognitive Systems and Materials Center), University of Groningen (Univ Groningen), Nijenborgh 4, NL-9747 AG Groningen, Netherlands}
\email{s.a.donderwinkel@rug.nl}
\author[Kerriou]{C\'eline Kerriou}
\address{Department of Mathematics and Computer Science, Universit\"at zu K\"oln, Weyertal 86-90, 50931 Cologne, Germany}
\email{ckerriou@math.uni-koeln.de}
\author[Lugosi]{G\'abor Lugosi}
\address{Department of Economics and Business, Pompeu
  Fabra University, Barcelona, Spain;
ICREA, Pg. Lluís Companys 23, 08010 Barcelona, Spain;
Barcelona Graduate School of Economics}
\email{gabor.lugosi@gmail.com}

\maketitle

\begin{abstract}
    We study a random recursive tree model featuring complete redirection called the random friend tree and introduced by \citet{SaramakiRFT}. Vertices are attached in a sequential manner one by one by selecting an existing target vertex and connecting to one of its neighbours (or friends), chosen uniformly at random. This model has interesting emergent properties, such as a highly skewed degree sequence. In contrast to the preferential attachment model, these emergent phenomena stem from a local rather than a global attachment mechanism.  The structure of the resulting tree is also strikingly different from both  the preferential attachment tree and the uniform random recursive tree: every edge is incident to a macro-hub of asymptotically linear degree, and with high probability all but at most $n^{9/10}$ vertices in a tree of size $n$ are leaves. We prove various results on the neighbourhood of fixed vertices and edges, and we study macroscopic properties such as the diameter and the degree distribution, providing insights into the overall structure of the tree. We also present a number of open questions on this model and related models.
\end{abstract}


\section{Introduction}\label{sec:intro}

{\bf Growing networks.} Various real-life phenomena, including contagion, social networks, rumour spreading, and the internet, have been described by models of growing networks (see e.g.\ \citet{kumar2020analysis}). Among these models, preferential attachment, introduced by \citet{doi:10.1126/science.286.5439.509} is arguably the most well-studied. In this model, vertices arrive one by one, and at each time a new vertex connects to one or more existing vertices with probability proportional to their degree. The degree sequence of this model satisfies the so-called scale-free property, which is often also observed in real-world models. In contrast to models such as the configuration model or the inhomogeneous random graph model, this property of the degree sequence is an intrinsic result of the dynamics rather than a pre-imposed characteristic. This makes the preferential attachment and its related models popular tools for understanding why real-world models may develop in this way. These models however require the knowledge of the full degree sequence in order to attach a new vertex. This requirement is unnatural for real-world networks and impractical in implementation.

{\bf The model. }The friend tree is a randomly growing network of which the dynamic also autonomously produces highly skewed degree sequences,  but whose attachment rule is based on redirection and requires only local information. Models involving redirection were introduced by \citet{KlKuRaRaTo1999} in directed, rooted graphs. In these models, a new vertex connects to a uniformly random vertex, or, with probability $p$, it connects to the ancestor of a randomly selected vertex. This mechanism, called {\em directed redirection}, yields a shifted linear preferential attachment rule, where the new vertex connects to a vertex with degree $d$ with probability proportional to $d-2+1/p$. A variant was studied by \citet{banerjee2022co}, where a new vertex attaches to the graph by randomly sampling a vertex and attaching to the endpoint of a path of random length directed towards the root. The undirected version of the model that we study was introduced by \citet{SaramakiRFT} and yields strikingly different graphs. In the works of \citet{SaramakiRFT} and \citet{EvansRFT}, the authors make the claim that the tree has the same law as a preferential attachment tree. This turns out to be inaccurate, as was noted by \citet{CanningsRFT}. In the undirected version, newly added vertices connect to a neighbour of a randomly selected vertex. More precisely, the starting tree $T_2$ consists of a single edge joining vertices labelled $1$ and $2$. Inductively, for $n \ge 2$, let $V_{n} \in \{1,\ldots,n\}$ be chosen uniformly at random and let $W_{n}$ be a uniformly random neighbour of $V_{n}$ in $T_{n}$. Then build $T_{n+1}$ from $T_{n}$ by attaching a new vertex labelled $n+1$ to the vertex $W_{n}$, see Figure~\ref{fig:RednerProcessBis}. Note that, for all $n\geq 2$, the vertex set of $T_n$ is $\{1,\ldots,n\}$. Moreover, setting $W_1=1$, then the edge set of $T_n$ is $\{\{m+1,W_m\},1 \le m \le n-1\}$.
We call $T_n$ the \textit{$1$-step friend tree}, inspired by the following picture. In a $1$-step friend tree a person selects a random stranger and befriends a uniformly random friend of theirs. A $2$-step friend tree would correspond to, instead, befriending a uniformly random friend of the stranger's random friend. In most of this work, we refer to the \textit{$1$-step friend tree} simply as the \textit{random friend tree (RFT)}.

\begin{figure}[h]
    \centering
    \includegraphics[scale=0.7]{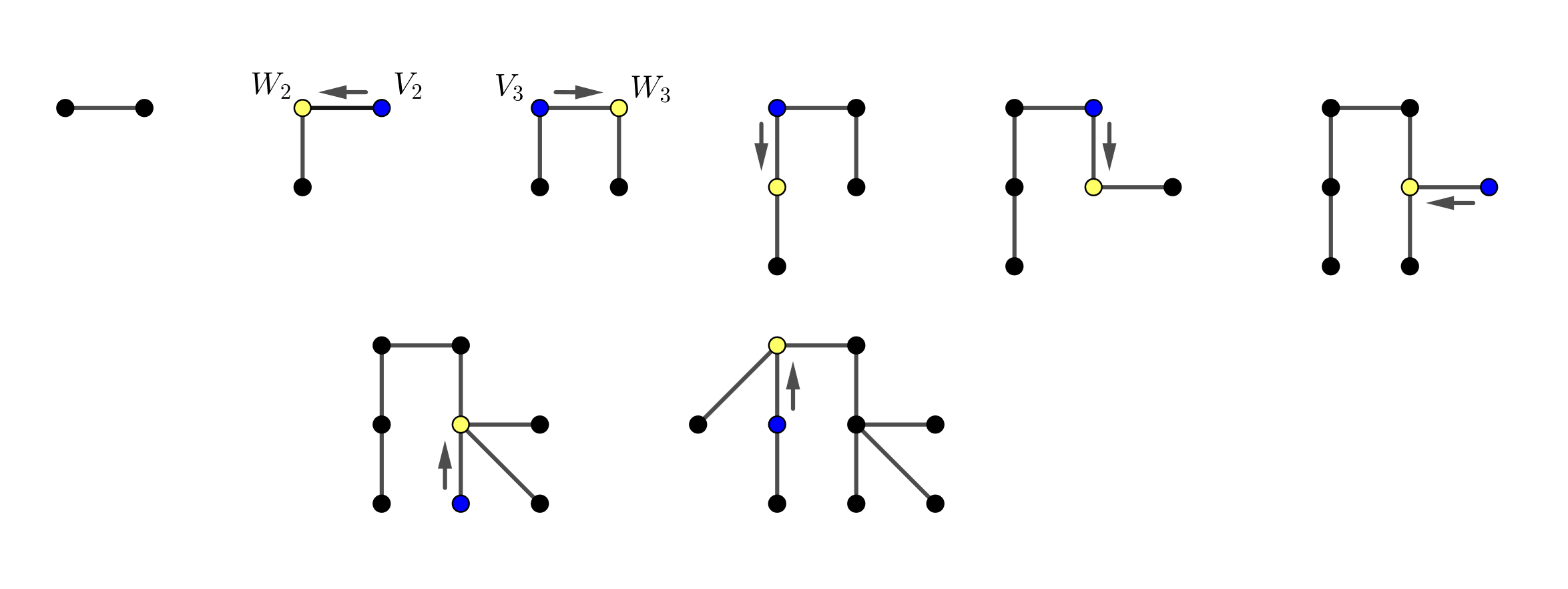}
    \caption{A realisation up to $n=9$, with $V_n$ in blue and $W_n$ in yellow.}
    \label{fig:RednerProcessBis}
\end{figure}

{\bf Motivation and challenges.} This model gives rise to numerous interesting emergent phenomena that make it worth studying. A first feature of the model is a rich-get-richer mechanism. In the preferential attachment model, it is embedded into the dynamics that vertices with many neighbours accumulate more neighbours. In the friend tree, vertices that have many neighbours \emph{with low degree} accumulate more neighbours. It turns out that, the highest degree vertices mostly have leaves as neighbours, so the growth of high-degree vertices is mostly determined by their degree. In fact, the reinforcement is strengthened by a second effect: the larger the degree of a vertex, the less likely it is for its neighbours to increase their degree, so the more likely it is for the high degree vertex itself to increase its degree. Unlike in the preferential attachment model, this is a result of the dynamics rather than a built-in feature. 

Furthermore, the model is part of a whole family of models that in some sense interpolate between the uniform random recursive tree (URRT) and the linear preferential attachment tree (PA tree). Indeed, in the $k$-step friend tree, new vertices attach to the endpoint of a random walk of length $k$ that starts at a random vertex. If $k=0$, the resulting model is the URRT. If $k$ was chosen so large that the random walk is perfectly mixed, the resulting model would be a PA tree. In other words, for $k$ large, the starting point of the random walk has a negligible effect and the end point is distributed proportionally to the degree of each vertex. Our work demonstrates several features of the $1$-step friend tree which are remarkably different from both the URRT and the PA tree. It is therefore worth investigating what range of behaviour can be observed in the entire family. In a related model, \citet{englander2023structural} study a ``random walk tree builder'' where a tree is grown by a random walk. More precisely, a walker is moving at random on the tree and at each time step $n$, with probability $n^{-\gamma}$, a neighbour is added to the vertex where the walker is. They prove that this model is actually a PA tree for an appropriate choice of $\gamma$. Unlike the friend tree model, where a new random walk is started at every time step, a single random walk is able to produce a tree displaying a rich gets richer phenomenon. This further motivates the study of the friend tree model in its whole range, and suggests the investigation of its possible links to the ``random walk tree builder'' model.

The challenges of studying $1$-step friend trees are numerous. For example, even if the process grows locally (one needs to know the neighbours of the randomly picked vertex to understand the connection probabilities), tracking only local information does not suffice to study how degrees evolve over time. Indeed, to understand how the vertex degrees change within two time steps, one must keep track of the degrees of second neighbours, and in general, for $t$ time steps one must track the $t$'th neighbourhood of every vertex. Note that this challenge does not arise in either directed redirection or the preferential attachment model; in those models the growth of a vertex degree only depends on the vertex degree itself, so degrees can be tracked on their own without considering the global structure of the tree. For $k$-step friend trees, the dependencies grow stronger as $k$ increases, bringing in new challenges that we do not attempt to tackle in the current work.

{\bf Results and comparison. }Some of the structural properties of random friend trees are comparable to those of the URRT. In Theorem \ref{thm:Diameter}, we show that the diameter is of logarithmic order almost surely (like for the URRT and the PA tree (\cite{Pittel_height_trees})). Moreover, as we show in Theorems \ref{thm:LeafDepth} and \ref{thm:LeafDepthURRT} respectively, both in the random friend tree and in the URRT, the largest distance to the nearest leaf in the $n$-vertex tree is $\Theta(\log(n)/\log\log(n))$ in probability. 

However, the interaction between neighbouring vertices in the attachment procedure yields significant structural differences between the random friend tree and both the URRT and the PA tree. The degree sequence might be the most illustrative of this difference. Regarding the high-degree vertices, we prove in Theorem \ref{thm:AbundanceHubs} that ``hubs'' of linear degree appear almost surely, whereas in a URRT the maximum degree is logarithmic (\citet{devroye1995strong}) and in linear preferential attachment tree the largest degree is of order $\sqrt{n}$  (\citet[Theorem 1.17]{hofstad_2016}). In fact, the dynamics of low-degree vertices `feeding their neighbours' implies that for every edge, at least one of the endpoints has asymptotically linear degree almost surely, so that a highly modular network emerges. This phenomenon has not been observed in any other random tree model, as far as the authors of this paper are aware. The existence of linear degree hubs also implies that two uniformly random vertices in $T_n$ are at distance two from each other with probability bounded away from zero, while in both the URRT and the PA tree typical distances grow logarithmically \cite{Luc_distance_URRT},\cite[Theorem 8.1]{hofstad_v2}.  As for low-degree vertices, both in PA trees and URRT an asymptotically positive fraction of vertices have degree at least two, but we shall show that a random friend tree of size $n$ has $n-o(n^{0.9})$ leaves. While a URRT of size $n$ has on average $n/2^{k-1}$ vertices of degree at least $k$ for fixed $k$ (\citet{janson2005asymptotic}), for friend trees, asymptotically, this number sits between $n^{0.1}$ and $n^{0.9}$ (see Theorem \ref{thm:SmallDegreeVertices}). The proliferation of hubs and the interaction between neighbours block most leaves from ever growing their degree. Proposition \ref{prop:mortal_leaves} shows that most leaves remain leaves forever\footnote{So one might argue that `random loneliness tree' is in fact a more appropriate name for our model.}, whereas for URRT and PA trees, the degree of every vertex is a.s.\ unbounded.

{\bf Earlier work. }The only previous rigorous result on random friend trees the authors are aware of was obtained by \citet{CanningsRFT}, who show that in the $1$-step friend tree, $n-o(n)$ of the vertices are leaves almost surely. Random friend trees were also studied in the physics literature by \citet{MR3683821}. In that work, the authors use simulations and non-rigorous arguments to study the distribution of size of the largest degrees and the order of growth of the number of non-leaves, and estimate the degree distribution restricted to the bounded-degree vertices. They conjecture that, for any fixed $k$, the number of degree $k$ vertices is of the order of $n^{\mu}$ for $\mu\approx 0.566$. Moreover, they conjecture that among the non-leaf vertices, the proportion of degree $k$ vertices is of the order of $k^{-(1+\mu)}$. We discuss their estimates in Section \ref{sec:questions}.

{\bf Outline.} First, in Section~\ref{sec:Notations} we introduce notation that we use throughout the paper. We then present our main results for friend trees. Our findings can naturally be divided into local and global properties, which we present in Section~\ref{sec:localproperties} and~\ref{sec:globalproperties} respectively. In Sections~\ref{sec:prooflocalproperties} and \ref{sec:proofglobalproperties} we prove our main results. Finally, Section~\ref{sec:questions} contains some open questions about random friend trees.

\begin{figure}[ht]
    \centering
    \includegraphics[scale=0.7]{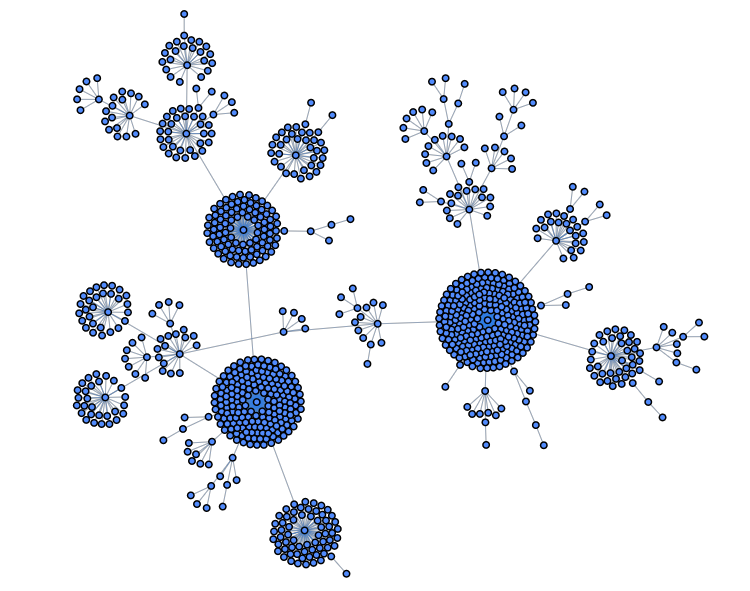}
    \caption{A realisation of $T_n$ with $n=1000$.}
    \label{fig:RednerProcess}
\end{figure}

\section{Notation}\label{sec:Notations}
For a graph $G$ and a vertex $v$ of $G$, write $\cN(v;G)$ for the neighbourhood of $v$ in $G$ and $\cL(v; G)$ for the set of leaf neighbours of $v$ in $G$ (i.e. vertices of degree one in $\cN(v;G)$). In the rest of the paper, $T_n$ denotes a tree of size $n$ obtained from the random friend tree model. The set of vertices of $T_n$ is $[n]:= \{1,\ldots,n\}$, where the label of the vertex is  its time of arrival in the tree. Since every integer $k \in \N$ is a vertex of $T_n$ for all $n \ge k$, we take the liberty of referring to integers of $\mathbb{N}$ as vertices. For $v \in T_n$, let $D_n(v)=|\cN(v;T_n)|$ be the degree and let $L_n(v) = |\cL(v; T_n)|$ be the number of leaf neighbours of $v$ in $T_n$. Define the random variable  
$$Z_v:=\liminf_{n\to\infty} \frac{D_n(v)}{n}~.$$
A vertex $v\in \N$ is called a \emph{hub} if $Z_v>0$, that is, if the degree of $v$ is of linear order asymptotically. A vertex $w$ is said to be a \emph{child} of vertex $v$ in $T_n$ if $w\in \cN(v;T_n)$ and $v<w$. If $w$ is a child of $v$, then $v$ is the \emph{parent} of $w$. For $i,j\in[n]$, denote by $d_n(i,j)$ the graph distance between vertices $i$ and $j$ in $T_n$. We also introduce $\mathrm{Diam}_n:=\max_{i,j\in[n]}d_n(i,j)$, the diameter of $T_n$, and let $M_n:=\max_{i\leq n} \min_{\{\ell: \ D_n(\ell)=1\}}d_n(i,\ell)$ be the maximal distance of any vertex of $T_n$ to its nearest leaf. 
For integers $n,k\ge 1$ we let $X_n^k=\{v \in [n]: D_n(v)=k\}$ be the number of vertices of degree $k$ in $T_n$, and let $X_n^{\ge k}=\sum_{j \ge k} X_n^j$.

For any sequence $(x_n)_{n\geq 1}$, we define, for all $n\geq 1$, $\Delta x_n := x_{n+1} - x_n$. For non-negative $(x_n)_{n\geq1}$ and positive $(y_n)_{n\geq 1}$ we write $x_n=O(y_n)$ and $y_n=\Omega(x_n)$ if $\limsup_{n\to\infty}\tfrac{x_n}{y_n}<\infty$ and we write $x_n=o(y_n)$ and $y_n=\omega(x_n)$ if $\lim_{n\to\infty} \tfrac{x_n}{y_n}=0$. We say that $x_n=\Theta(y_n)$ if $x_n=O(y_n)$ and $x_n=\Omega(y_n)$ both hold. We also use this notation with a $p$ subscript meaning that the property holds in probability. So, for sequences $(X_n)$ and $(Y_n)$ of non-negative random variables, we write $X_n=O_p(Y_n)$ and $Y_n=\Omega_p(X_n)$ if for all $\varepsilon>0$, there exist $M>0$ and $K>0$ such that for all $n\geq M$, $\p{X_n>KY_n }<\varepsilon$. We write $X_n=o_p(Y_n)$ and $Y_n=\omega_p(X_n)$ if for all $\varepsilon>0$, $\delta>0$, there exists $M>0$ such that for all $n\geq M$, $\p{X_n>\delta Y_n }<\varepsilon$. We write  $X_n=\Theta_p(Y_n)$ if both $X_n=O_p(Y_n)$ and $X_n=\Omega_p(Y_n)$.

\section{Local results}\label{sec:localproperties}
In this section we state our results regarding the properties of individual vertices and their close neighbourhoods. First, we state a convergence result for the normalised degree. 
\begin{tcolorbox}
    \begin{thm}[Convergence of normalised degree]\label{thm:DegreeConvergence}
        For vertex $u\in\mathbb{N}$, the random variables $D_n(u)$, $L_n(u)$ and $Z_u$, defined in Section \ref{sec:Notations}, are such that 
        \begin{align*}
            \frac{D_n(u)}{n}\to Z_u  & \text{ and } \frac{L_n(u)}{n}\to Z_u 
        \end{align*}
     almost surely as $n\to \infty$.
    \end{thm}
\end{tcolorbox}
We observe that $\sum_{i\ge 1}Z_i\le1$ almost surely. Indeed, on the event that this sum exceeds $1$, then there must be $k \in \N$ and $\delta>0$ such that $Z_1+\dots+Z_k=1+\delta$. But this would imply that there is a finite $n$ such that the number of leaves neighbouring vertices $1,\dots,k$ in $T_n$ satisfies $L_n(1)+\dots+L_n(k)\ge (1+\delta/2)n$. However, the number of leaves in $T_n$ is at most $n-1$ deterministically, so this gives a contradiction.  

We conjecture that $\sum_{i\ge 1} Z_1=1$ almost surely, implying that, asymptotically, all but a negligible proportion of the vertices are a leaf next to a hub. We discuss this and other open questions on the law of $(Z_i)_{i\ge 1}$ in Section~\ref{sec:questions}.
    
A striking property of the friend tree concerns the degree of an edge. 

\begin{tcolorbox}
    \begin{thm}[Abundance of hubs]\label{thm:AbundanceHubs}
        The degree of every edge is almost surely asymptotically linear. That is, for any $m\geq 1$, if $W_m$ is the parent of $m+1$, then $Z_{m+1}+Z_{W_m}  >0$ almost surely.
    \end{thm}
\end{tcolorbox}

We also study the limit of the proportion of vertices that are adjacent to $W_n$. This theorem shows that the mean of the empirical law of $D_n(W_n)/n$ given $T_n$ converges almost surely, implying that this global property of $T_n$ ``stabilizes" as $n$ grows large. We conjecture that, in fact, the empirical law itself converges almost surely to $\sum_{i \ge 0} Z_i\delta_{Z_i}$ with respect to the Prokhorov topology. This is in fact equivalent to the conjecture that $\sum_{i \ge 1} Z_i = 1$ almost surely.

\begin{tcolorbox}
    \begin{thm}[Expected degree of $W_n$]\label{thm:deg_W_n}
        As $n\to \infty$, $\tfrac{1}{n}\E{D_n(W_n)|T_n}$ has a positive almost sure limit. Moreover, $\tfrac{1}{n}\E{D_n(W_n)}$ converges to some positive number as $n\to\infty$.   \end{thm}
\end{tcolorbox}

Another notable property of random friend trees concerns the probability of a vertex having bounded degree.

\begin{tcolorbox}
    \begin{thm}[Bounded degree]\label{thm:StationaryDegree}
       Let $k$ be a positive integer. Fix $v\in\mathbb{N}$. Then, for all $n\geq v$,
    \[\probC{D_{n+j}(v) = k \ \forall j\geq 0}{T_n} > c_k\I{D_n(v)=k},\]
    where the constant $c_k>0$ only depends on $k$.
    \end{thm}
\end{tcolorbox}
\noindent  In particular, Theorem \ref{thm:StationaryDegree} implies that it is impossible to `diagnose' which vertices are the hubs, even at a very large time. Indeed, every vertex, no matter how large its degree is, has probability bounded away from zero to never acquire any new neighbours.
We prove Theorems~\ref{thm:DegreeConvergence} and \ref{thm:AbundanceHubs} in Section~\ref{sec:Hubs}, the proof of Theorem~\ref{thm:deg_W_n} can be found in Section~\ref{sec:deg_W_n} and  Theorem~\ref{thm:StationaryDegree} is proven in Section~\ref{sec:EternalLEaves}.

\section{Global results}\label{sec:globalproperties}
We now present our results on the global properties of random friend trees.

\begin{tcolorbox}
\begin{thm}[Typical distances]
\label{thm:typical}
    For $U_n$ and $V_n$ two uniformly random vertices in $T_n$, it holds that the distance between $U_n$ and $V_n$ is equal to $2$ with probability bounded away from zero. Moreover, the distance between vertex $1$ and $U_n$ is at most $2$ with probability bounded away from $0$.
\end{thm}
\end{tcolorbox}

The previous result sets the random friend tree apart from the ``universality class" of logarithmic trees (and in particular from the uniform random recursive tree and preferential attachment trees), in which typical distances are logarithmic.

The next result shows that, while distances between typical vertices can be very small, the diameter of random friend trees is indeed logarithmic. 

\begin{tcolorbox}
    \begin{thm}[Diameter of random friend trees]\label{thm:Diameter}
        Almost surely
        $$1\leq \liminf_n \frac{\mathrm{Diam}_n}{\log(n)} \leq \limsup_n \frac{\mathrm{Diam}_n}{\log(n)} \leq 4e~. $$
    \end{thm}
\end{tcolorbox}
In particular, this means that, asymptotically, a path of arbitrary length is present in the tree. Since for every edge at least one of its endpoints is a hub, this implies that an asymptotically unbounded number of hubs are present. We strengthen this statement in Theorem~\ref{thm:numberofhubs}, where we prove an almost sure polynomial lower bound for the number of hubs in $T_n$.

We also study the leaf-depth in $T_n$. The next result implies that, although each vertex is at distance at most $1$ from a hub, and will therefore eventually be at distance at most two from the nearest leaf, at fixed times there are still exceptional locations in the graph where the nearest leaf is much further away.  
\begin{tcolorbox}
    \begin{thm}[Leaf depth]\label{thm:LeafDepth}
        Let $M_n$ be the maximal distance of any vertex to its nearest leaf in $T_n$. Then 
        $$ M_n=\Theta\left( \frac{\log(n)}{\log\log(n)} \right)~\text{in probability}. $$
    \end{thm}
\end{tcolorbox}
An input to the proof of Theorem \ref{thm:LeafDepth} is the corresponding result for the URRT, which we state as a separate theorem.

\begin{tcolorbox}
    \begin{thm}[Leaf depth in URRT]\label{thm:LeafDepthURRT}
        Let $M'_n$ be the maximal distance of any vertex to its nearest leaf in a URRT. Then, 
        $$ M'_n=\Theta\left( \frac{\log(n)}{\log\log(n)} \right)~\text{in probability}. $$
    \end{thm}
\end{tcolorbox}
We prove Theorem~\ref{thm:LeafDepth} using Theorem~\ref{thm:LeafDepthURRT} and a coupling between URRTs and random friend trees under which distances are at most a factor of two larger in the URRT than in the random friend tree to which it is coupled. This coupling is presented in Lemma~\ref{lem:coupling_URRT_1}, below.

Another global property of interest is the degree distribution of the tree. We study degree statistics at both ends of the spectrum, for both sub-linear-degree vertices and bounded-degree vertices. We get the following lower bound on the number of hubs in $T_n$. 
\begin{tcolorbox}
\begin{thm}[Number of hubs]
\label{thm:numberofhubs}
There exists a constant $\delta>0.1$ such that 
\[ \frac{\#\{u\in [n]:Z_u>0\}}{n^{\delta}}\to \infty\text{ a.s.}\]
\end{thm}
\end{tcolorbox}

The following theorem concerns the number of high degree vertices.
\begin{tcolorbox}
    \begin{thm}[Abundance of high degree vertices]\label{thm:InfinityLinearDegree}
        For any sequence $(m_n)_{n\geq 1}$, satisfying $m_n=o(n)$, almost surely
        $$ \lim_{n\to \infty} X_n^{\geq m_n}=\infty.$$
    \end{thm}
\end{tcolorbox}

The next theorem gives polynomial upper and lower bounds for the number of bounded degree vertices. 
\begin{tcolorbox}
    \begin{thm}[Polynomial bounds on low-degree vertices]\label{thm:SmallDegreeVertices}
        There exist constants $0.1<\delta \leq \lambda<0.9$ such that, for any $k\geq 2$, almost surely
        $$ \lim_{n\to \infty } \frac{X_n^{\geq k}}{n^{\delta}}=\infty~, $$
        $$ \lim_{n\to \infty } \frac{X_n^{\geq k}}{n^{\lambda}}=0~. $$
    \end{thm}
\end{tcolorbox}
It has been conjectured by \citet{MR3683821} that a stronger statement is true for the RFT. They conjecture that there exists a constant $\mu\approx 0.566$ such that, $n^{-\mu}X_n^{\geq k}\to X_k$, where $X_k$ is a non-degenerate random variable. Moreover, they conjecture that $X_n^{ k}/X_n^{\geq 2}$ has an almost sure limit which is $\Theta(k^{-(1+\mu)})$.

Finally, we show that $X_n^{\ge k}=\Theta(X_n^{\ge 2})$ almost surely, for any fixed $k$. 

\begin{tcolorbox}
    \begin{thm}[Comparing low-degree vertices]\label{thm:all_finite_deg_large}
There exists a sequence $(c_k)_{k\geq 2}$ of positive real numbers, such that for any $k\geq 2$,  
\[c_k<\liminf_{n\geq \infty} \frac{X_n^{\geq k+1}}{ X_n^{\geq k}} \le 1 \]
almost surely.
\end{thm}
\end{tcolorbox}

We prove Theorem~\ref{thm:typical} in Section~\ref{sec:typical} and we prove Theorem~\ref{thm:Diameter} in Section~\ref{sec:Diameter}. The proofs of Theorems~\ref{thm:LeafDepth} and \ref{thm:LeafDepthURRT} are in Section~\ref{sec:Leafdepth}. Finally, Theorems~\ref{thm:numberofhubs} and \ref{thm:InfinityLinearDegree} are proven in Section~\ref{sec:LargeDegreeVertices} and Theorems~\ref{thm:SmallDegreeVertices} and \ref{thm:all_finite_deg_large} are proven in Section~\ref{sec:SmallDegreeVertices}. Although Theorem~\ref{thm:SmallDegreeVertices} is invoked in the proofs of several of the earlier results, we postpone its proof to later in the paper, as it is quite technical. 

\section{Proofs of local properties}\label{sec:prooflocalproperties} 
\subsection{Hubs}[Proof of Theorems~\ref{thm:DegreeConvergence} and ~\ref{thm:AbundanceHubs}]\label{sec:Hubs} 
We prove Theorem \ref{thm:DegreeConvergence} with a submartingale argument, deferring a crucial step of the proof to Section \ref{sec:proofglobalproperties}.

\begin{proof}[Proof of Theorem~\ref{thm:DegreeConvergence}]
Fix $v\in T_n$. Note that $L_{n+1}(v) = L_n(v) + 1$ if  $W_n = v$. For $V_n = u$ a neighbour of $v$, $\p{W_n=v|V_n=u}=1/D_n(u)$. Since $V_n$ is a uniform sample from $[n]$, it follows that
$$ \p{L_{n+1}(n)=L_n(v)+1\mid T_n}=\sum_{u \in \cN(v;T_{n})} \frac{1}{n} \frac{1}{D_{n}( u)}~. $$
Next, $L_{n+1}(v) = L_n(v) -1$ if $V_n= v$ and $W_n \in \cL(v; T_n)$. Thus,

$$\Cexp{L_{n+1}(v)}{T_{n}} =L_{n}(v) - \frac{1}{n} \frac{L_{n}(v)}{D_{n}(v)} + \sum_{u \in \cN(v;T_{n})} \frac{1}{n} \frac{1}{D_{n}( u)}~.$$
Using that $L_{n}(v) \le D_{n}(v)$ and, for a leaf $v$, $D_{n}(v)=1$, we can lower bound the above by $L_{n}(v) - \frac{1}{n} + \frac{L_{n}(v)}{n}$. By rearrangement it follows that

\begin{align*}
\Cexp{\frac{L_{n+1}(v)-1}{n+1}}{T_{n}}  \ge \frac{L_{n}(v)-1}{n}. 
\end{align*}
Thus, for any $m \in \N$, the process $((L_n(v)-1)/n,n \ge m-1)$ is a submartingale relative to the filtration generated by the random friend tree process. It is bounded, so it converges almost surely. Furthermore, by the trivial inequalities
    \[\frac{L_n(u)}{n}\leq  \frac{D_n(u)}{n}\leq  \frac{L_n(u)+X_n^{\geq 2}}{n}~,\]
    the joint convergence follows from Theorem \ref{thm:SmallDegreeVertices} below, stating that $n^{-1}X_n^{\geq 2}\to 0$ almost surely.
\end{proof}

Recall that a vertex $u$ is a hub if $Z_u>0$. Theorem \ref{thm:AbundanceHubs} implies that each edge has at least one endpoint that is a hub, which, in particular, shows that hubs exist.

\begin{proof}[Proof of Theorem~\ref{thm:AbundanceHubs}]
Fix $m\in\mathbb{N}$. For $n\geq m+1$, write $D_n:=D_n(m+1)+D_n(W_m)$ for the total number of neighbours of vertex $m+1$ and vertex $W_m$ at time $n$. Write $L_n:=L_n(m+1)+L_n(W_m)$ for the number of those neighbours that are leaves. We have $L_{m+1}\geq 1$ and $D_{m+1}\geq 3$. Note that $D_n$ is non-decreasing and that if $D_{n+1}=D_n+1$ then also $L_{n+1}=L_n+1$, so $\Delta(L_n,D_n)\in \{(1,1),(0,0),(-1,0)\}$. 

Moreover,
\begin{equation}\label{eq:HubIncreaseLowerBound}
\probC{\Delta(L_n,D_n)=(1,1)}{(L_i,D_i),m+1 \le i \le n}
\ge 
\frac{1}{n}\pran{L_n+\frac{1}{D_n}}\, ,   
\end{equation}
since, to have $\Delta(L_n,D_n)=(1,1)$, it suffices that either $V_n\in\cL(m+1;T_n)\cup \cL(W_m;T_n)$ or else that 
$\{V_n,W_n\}=\{m+1,W_m\}$. We also have
\begin{equation}\label{eq:HubDecreaseUpperBound}
   \probC{\Delta(L_n,D_n)=(-1,0)}{(L_i,D_i),m+1 \le i \le n} \le \frac{\min(2,L_n)}{n}\, , 
\end{equation}
since if $\Delta(L_n,D_n)=(-1,0)$, then $V_n \in \{m,W_m\}$ and $W_n\in\cL(m;T_n)\cup \cL(W_m;T_n)$. 

We claim that $D_n\to \infty$ almost surely. We fix $k\in \N$ and show that the stopping time $\tau_{k}=\min\{n\ge m+1:D_n\ge k\}$ is finite almost surely. On the event that $\tau_k<n$, \eqref{eq:HubIncreaseLowerBound} implies that $\Delta D_n\mid T_n$ stochastically dominates a $\operatorname{Bernoulli}(\tfrac{1}{nk})$ random variable. If $\mathcal{B}_i$ are independent  $\operatorname{Bernoulli}(\tfrac{1}{ik})$ random variables, $ \sum_{i=m+1}^n \mathcal{B}_i \to \infty $ almost surely, so $\tau_k<\infty$ almost surely. It then follows that $\#\{n:L_n>0\}=\infty$ almost surely. Indeed, $D_n\to \infty$ implies that  $\Delta(D_n,L_n)=(1,1)$ infinitely many times, so if $L_n$ ever hits zero it almost surely becomes positive again. 

Let $(J_k,k \ge 0)$ be the sequence of jump times of the process $(\Delta(L_n, D_n))_n$, that is $J_0= m+1$ and $J_k = \min\{\ell > J_{k-1}:\Delta(L_{\ell - 1}, D_{\ell -1}) \neq (0,0)\}$ for $k\geq 1$. This sequence has infinite length because $D_n\to \infty$ almost surely.

Our goal is to show that $L_n$ grows linearly. To do so, we couple $L_n$ to an urn process. Inequalities \eqref{eq:HubIncreaseLowerBound} and \eqref{eq:HubDecreaseUpperBound} suggest that the growth of $L_n$ is similar to the growth of the number of black balls in a standard P\'olya urn with black and white balls. One difference being that, at time $n$, with probability at most $2/n$, a black ball is replaced by a white ball. Nonetheless, we exhibit a coupling between $L_n$ and  the number of black balls in a standard P\'olya urn of black and white balls of size $n$. More precisely, a coupling where $L_n$ is greater than the number of black balls in the P\'olya urn of size $n$. This coupling can fail, meaning that it is valid until a random time $S$ that is finite with positive probability. We say that the coupling succeeds if $S=\infty$. We show that this coupling succeeds with probability greater than $0$, and that, if it fails, we may try again by starting a new coupling at a subsequent time. This guarantees that one of the coupling attempts is successful, proving that $L_n$ grows linearly, because the number of black balls in a standard P\'olya urn grows linearly almost surely.  

The coupling is started at a time where $L_n$ is at least $10$. The bounds \eqref{eq:HubIncreaseLowerBound} and \eqref{eq:HubDecreaseUpperBound} on the transition probabilities show that the process $(L_{J_k},k \ge 0)$ stochastically dominates a  simple symmetric random walk reflected at $0$. This, in particular, implies that there are infinitely many $n$ such that $L_n\geq 10$, because $J_k\to \infty$ almost surely.

Let $\rho_1$ be the first time for which $L_{\rho_1}\geq 10$ (and so $L_{\rho_1}=10$). A first coupling is started from time $\rho_1$. For any $k$, if the $k$th coupling fails, let $\rho_{k+1}$ be the first time after the failure at which $L_{\rho_{k+1}}\geq 10$ and start the $(k+1)$st coupling from that time. We show that there is a $c>0$ so that for each $k$, given that couplings $1,\ldots,k-1$ all failed, the $k$th coupling succeeds with probability at least $c$. This implies that there is an almost surely finite $M$ so that the $M$th coupling succeeds.

So let us fix some $N>0$ and condition on $\rho_k=N$. Now define $B_n$ the number of black balls in a Pólya urn starting at time $N$ with $5$ black balls and $N-5$ white balls (with the standard replacement rule that a drawn ball is replaced along with one extra ball of the same colour). Note that $L_N=L_{\rho_k}\geq 10$ so $L_N-B_N\geq 5$. We couple $L_n$ and $B_n$ from time $N$ onwards and we say the coupling fails at time $S$ if $S$ is the first time $S>N$ at which $L_S-B_S\leq 4$. If the coupling never fails we set $S=\infty$. Then, for $n\in[N,S]$, we can couple $L_n$ and $B_n$ such that if $B_{n+1}=B_n+1$ then $L_{n+1}=L_n+1$. From this coupling, for $n\ge N$,
\begin{align*}
    \p{\Delta (L_{n}-B_{n})=1\mid S>n }&\ge  \frac{L_n-B_n}{n}\\
    \p{\Delta (L_{n}-B_{n})=-1\mid S>n}&\le  \frac{2}{n}. 
\end{align*}
This means that until the coupling fails, $L_n-B_n$ can be coupled to a symmetric random walk for which an increment with value $1$ is twice as likely as an increment with value $-1$.  We introduce $R_k$, a random walk with $R_0=5$ and
$$R_{k+1}-R_k=\begin{cases}
    +1 \text{ with probability } \frac{2}{3}, \\
    -1 \text{ with probability } \frac{1}{3}~.
\end{cases}$$
Set $I_0=N$ and let $I_{k+1}=\min\{j\ge I_k:\Delta(L_n-B_n)\neq 0\}$ be the $k$th jump time of $L_n-B_n$. Then $(R_k,k\ge 0)$ and $(L_n-B_n, n\ge N)$ can be coupled such that if $I_k\le S$
$$L_{I_k}-B_{I_k}\geq R_{k}.$$
With positive probability $R_k>4$ for all $k$, so with positive probability, not depending on $N$, $S=\infty$. This shows that, almost surely, one of the coupling attempt succeeds. Suppose that the $k$th coupling succeeds and that $\rho_k=N$. Then, for $B_n$ as above, $L_n\geq B_n$ for $n\geq N$, so since 
\[\lim_{n\to\infty}\frac{B_n}{n}>0\text{ almost surely}, \]
by a standard result on P\'olya urns (see \citet[Section 3.2]{mahmoud2008polya}), we also get that
\[\liminf_{n\to\infty}\frac{L_n}{n}>0\text{ almost surely,} \] which implies the statement.\qedhere
\end{proof}
\subsection{Expected degree of $W_n$}[Proof of Theorem~\ref{thm:deg_W_n}]\label{sec:deg_W_n}

We show the statement using the almost sure martingale convergence theorem by identifying a supermartingale. Set $Y_n:=\E{D_n(W_n)|T_n}$ so that $Y_n$ is adapted to $\sigma(T_n)$ and note that

\[Y_n=\frac{1}{n}\sum_{i\in [n]}\Cexp{D_n(W_n)}{T_n, V_n=i} =\frac{1}{n}\sum_{i\in [n]}\left( \frac{1}{D_n(i)}\sum_{j\sim_{T_{n}} i} D_n(j)\right)~.\]
To identify the supermartingale, we study

\[\Cexp{(n+1) Y_{n+1}}{ T_{n}}=\Cexp{\sum_{i\in [n+1]} \sum_{j\sim_{T_{n+1}} i}\frac{D_{n+1}(j)}{D_{n+1}(i)} }{T_n}.\]
Observe that the only randomness in $Y_{n+1}$, conditional on $T_n$, comes from the choice of $W_n$. It holds that $D_{n+1}(W_n)=D_{n}(W_n)+1$, and the new neighbour of $W_n$ (vertex $n+1$) has degree $1$. Moreover, $D_{n+1}(n+1)=1$, because every vertex starts as a leaf, and its neighbour is $W_n$. Finally, $D_{n+1}(i)= D_{n}(i)$ for all other $i$. Therefore, we get the following equalities for the different terms in $\sum_{i\in [n+1]} \sum_{j\sim_{T_{n+1}} i}\frac{D_{n+1}(j)}{D_{n+1}(i)}$:

\[ \sum_{j\sim_{T_{n+1}} i}\frac{D_{n+1}(j)}{D_{n+1}(i)}=\begin{cases}\frac{1}{D_n(W_n)+1} + \sum_{j\sim_{T_{n}} W_n}\frac{D_{n}(j)}{D_{n}(W_n)+1}&\text{ for }i=W_n\\
D_{n}(W_n)+1 &\text{ for }i=n+1\\
\frac{\I{i\sim_{T_n} W_n} }{D_{n}(i) } + \sum_{j\sim_{T_{n}} i}\frac{D_{n}(j)}{D_{n}(i)}&\text{ otherwise.}
\end{cases}\]
Combining these cases, we see that 
\begin{align*}
& \Cexp{(n+1) Y_{n+1}}{T_{n}}\\
&=\sum_{i\in [n]} \sum_{j\sim_{T_n} i} \frac{D_{n}(j)}{D_{n}(i)}\\
&\quad +\Cexp{\frac{1}{D_n(W_n)+1}-\left( \frac{1}{D_n(W_n)}-\frac{1}{D_n(W_n)+1}\right)\sum_{j\sim_{T_n} W_n} D_n(j)}{T_n}\\
&\quad +\Cexp{D_n(W_n)+1+\sum_{j\sim_{T_n} W_n} \frac{1}{D_n(j)} }{T_n}.
\end{align*}
Then, using that $D_n(j)\ge 1$ for all $j\sim_{T_n} i$ we get that the second term on the right hand side is positive. To get an upper bound for the third term, we again use that $D_n(j)\ge 1$ to get that

\[\Cexp{(n+1) Y_{n+1}}{ T_{n}} \le \sum_{i\in [n]} \sum_{j\sim_{T_n} i} \frac{D_{n}(j)}{D_{n}(i)}+\Cexp{2D_n(W_n)+1}{T_n}\le (n+2)Y_n +1,\]
so that $\Cexp{Y_{n+1}/(n+2)}{ Y_n} \leq Y_n/(n+1)+1/(n+1)^2$, and therefore $Y_n/(n+1)-\sum_{i=1}^n 1/i^2$ is a supermartingale in the filtration generated by $T_n$, and therefore has an almost sure limit. Since $1/i^2$ is summable, it follows that $\tfrac{1}{n}\Cexp{D_n(W_n)}{T_n}$ has an almost sure limit.

To see that the limit is positive, note that Theorem~\ref{thm:AbundanceHubs} implies that for any $\varepsilon>0$,  there is a $\delta>0$ so that $\p{Z_1+Z_2>\delta}>1-\varepsilon$.

Then, observe that 
\begin{align*}&\Cexp{\tfrac{1}{n}D_n(W_n)}{ T_n}>\tfrac{1}{n}D_n(1)\Cprob{W_n=1 }{ T_n}+\tfrac{1}{n}D_n(2)\Cprob{W_n=2 }{ T_n}\\
&\ge \frac{L_n(1)^2+L_n(2)^2}{n^2}~,\end{align*}
since $D_n(i)\geq L_n(i)$ and $\Cprob{W_n=i }{ T_n}\geq L_n(i)/n$ for $i=1,2$. Then, note that, if $L_n(1)+L_n(2)>\delta n/2$ then  \[\frac{L_n(1)^2+L_n(2)^2}{n^2}\ge \left(\max\{L_n(1)/n,L_n(2)/n\}\right)^2\ge (\delta/4)^2,\] so
\begin{align*}\p{\tfrac{1}{n}\Cexp{D_n(W_n)}{T_n}>\delta^2/16}\ge \p{L_n(1)+L_n(2)>\delta n/2} .\end{align*}
But, $\tfrac{1}{n}(L_n(1)+L_n(2))\to Z_1+Z_2$ almost surely, so \[\liminf_{n\to\infty} \p{L_n(1)+L_n(2)>\delta n/2}>1-\varepsilon\] so also \[\liminf_{n\to\infty} \p{\tfrac{1}{n}\Cexp{D_n(W_n)}{T_n}>\delta^2/16}>1-\varepsilon,\]  which implies the statement.

The convergence in expectation follows from the bounded convergence theorem, since $\tfrac{1}{n}D_n(W_n)\le 1$ deterministically.

\subsection{Eternal leaves and eternal degree  \texorpdfstring{$k$}{k} vertices}[Proof of Theorem~\ref{thm:StationaryDegree}]\label{sec:EternalLEaves}
Note that, for any integer $i$, $D_n(i)$ is increasing in $n$ and therefore it has an almost sure limit (that might be infinite). For a vertex $\ell$ that is a leaf at time $m$, we say it is \emph{temporary} if $\lim_{n\to \infty} D_n(\ell)>D_m(\ell)=1$. Otherwise we call it \emph{eternal}. Similarly, we call a vertex $v$ that has degree $k$ at time $m$ \emph{temporary} if $\lim_{n\to \infty} D_n(v)>D_m(v)=k$ and otherwise we call it \emph{eternal}. Informally, our next proposition says that, only a bounded number of leaves next to a given hub ever stop being a leaf. 

\begin{prop}\label{prop:mortal_leaves}
 For $n\geq v$, let $S_n(v)$ be the number of temporary leaves attached to $v$ at time $n$. If $\p{Z_v>0}>0$, then conditional on $Z_v>0$, $S_n(v)$ is tight, that is, for all $\varepsilon>0$ there exists a constant $M>0$ such that $\Cprob{S_n(v) > M}{Z_v >0}<\varepsilon$ for all $n\geq v$.  
\end{prop}

\begin{proof}

Fix $\varepsilon>0$. Suppose $v$ is a hub, that is $Z_v=\lim_{n\to\infty}\frac{D_n(v)}{n}>0$. This implies that there is a $\delta$ and a $N>v$ such that $\p{\forall n\ge N\;L_n(v)>\delta n \mid Z_v>0}>1-\varepsilon/2$. We show that there exists a constant $K$ such that $\E{S_k(v)\I{\forall n\ge N\; L_n(v)>\delta n }}<K$
for all $k\geq N$.  We first show that this implies the statement. Write $\p{Z_v>0}=\rho$, so that $\E{S_k(v)\I{\forall n\ge N\;L_n(v)>\delta n}\mid Z_v>0}<K/\rho$. Observe that $L_n(v)>\delta n \;\forall n\ge N$ implies $Z_v>0$. Then, we see that for $k\geq N$ and $\varepsilon >0$, 

\begin{align*} &\p{S_k(v)>\tfrac{2K}{\rho\varepsilon} \mid Z_v>0}\\
&\le \p{\exists n\ge N : L_n(v)\le \delta n \mid Z_v>0 } + \p{S_k(v)>\tfrac{2K}{\rho\varepsilon}, \forall n\ge N\;L_n(v)>\delta n \mid Z_v>0 } \\
&\le \varepsilon/2+ \p{S_k(v)>\tfrac{2K}{\rho\varepsilon} \mid \forall n\ge N\; L_n(v)>\delta n }\p{L_n(v)>\delta n \;\forall n\ge N  \mid Z_v>0 }\\
&\le \varepsilon/2+ \tfrac{\varepsilon\rho}{2K}\Cexp{S_k(v)}{ \forall n\ge N\; L_n(v)>\delta n }\p{L_n(v)>\delta n \;\forall n\ge N  \mid Z_v>0 } \\
&= \varepsilon/2+ \tfrac{\varepsilon\rho}{2K} \Cexp{S_k(v)\I{\forall n\ge N\; L_n(v)>\delta n }}{ Z_v>0}<\varepsilon,
\end{align*}
where we use Markov's inequality in the penultimate line. 

We now show that there exists a constant $K$ such that $\E{S_k(v)\I{ \forall n\ge N\; L_n(v)>\delta n } }<K$
for all $k\geq N$.
Note that, at any time $M\geq k$, if $w$ is a leaf neighbouring $v$, then $w$ stops being a leaf if $V_M=v$ and $W_M = w$. Conditionally on $T_M$, this occurs with probability $1/(M D_M(v))$, so the probability that this happens for some vertex in $\cL(v; T_k)$ is at most $k/(M D_M(v))$, because $|\cL(v; T_k)|<k$.  Therefore, the probability that the number of leaves in $\cL(v; T_k)$ that are no longer leaves  increases at time $M\geq k$ satisfies 
\begin{align*} &\p{\Delta|\cL(v; T_k) \cap  \cL(v;T_M)^c | =1 , L_n(v)>\delta n  \;\forall n\ge N }\\
&\leq \p{\Delta|\cL(v; T_k) \cap  \cL(v;T_M)^c | =1 , L_M(v)>\delta M} \leq\frac{k}{\delta M^2}.
\end{align*}
Therefore 
\[\E{S_k(v)\I{ L_n(v)>\delta n \;\forall n\ge N}}\leq \sum_{M\geq k}\frac{k}{\delta M^2}\leq K\]
for some constant $K$ not depending on $k$, which proves the claim.
\end{proof}

The next corollary, stating that the number of eternal leaves attached to any edge grows linearly with high probability, is a consequence of the proposition above and Theorem~ \ref{thm:AbundanceHubs}. Indeed, since every edge has linear degree almost surely, given the presence of edge $(u,v)$, either $u$ or $v$ has probability at least $1/2$ of being a hub.

\begin{cor}
Fix an edge $(u,v)$ and for $n\geq \max(u,v)$ define $E_n(u,v) := |\{w\in \cN(u;T_n)\cup \cN(v;T_n): w \text{ is an eternal leaf}\}|$. Then, $E_n(u,v)$ grows linearly with high probability, that is, there exists $ \delta >0$ such that for every $\varepsilon>0$ there is a $N=\cN(u,v)$ such that
$$ \p{\exists n>N: E_n(u,v)<\delta n} <\varepsilon ~. $$
\end{cor}

We now prove Theorem~\ref{thm:StationaryDegree}.

\begin{proof}[Proof of Theorem~\ref{thm:StationaryDegree}] 
    Fix $v\in \N$. We first prove the statement for $k=1$, and then discuss how to adapt the proof to general $k$. If $D_n(v)=1$ then let $w$ denote the unique neighbour of $v$ in $T_n$. We show that $v$ is an eternal leaf with positive probability by showing that, with positive probability, the vertex $w$ acquires a large number of leaf neighbours, ensuring that both the degree of $w$ grows and that, when a new vertex is attached to a uniform neighbour of $w$, it is unlikely that $v$ is chosen.
    Fix $N$ and let    
    \[\tau=\min\left\{t\in \N:\#\{n<i\leq t: V_i\in \{v,w\}\right\}=N\}\] 
    be the random time at which a new vertex is attached to a random neighbour of either $v$ or $w$ exactly $N$ times since time $n$. Define $A_N$ as the event $\{\#\{n<i\leq \tau: V_i =v\}=N\}$. Since $V_i$ is chosen uniformly at random, with probability $2^{-N}$, $V_i=v$ exactly $N$ times between times $n$ and $\tau$. So, $\p{A_N}=2^{-N}$. Recall that $D_n(v)=1$, so conditionally on $A_N$,  $D_\tau(v)=1$ because for all $i\in[n,\tau]$, $V_i\neq u$, and thus $W_i\neq v$. Moreover, $L_\tau(w)\geq N$, because, conditioned on $A_N$, when $V_i=v$, $W_i=w$, so that a new leaf is attached to $w$. Since for all $i\in[n,\tau]$ $V_i\neq u$, these leaves stay leaves until time $\tau$. We show that, on the event $A_N$, $v$ is an eternal leaf with positive probability, because it is likely that $w$ continues to acquire many leaf neighbours beyond time $\tau$, making it unlikely for the degree of $v$ to grow. 

    For $j\geq 0$, set $a_j=\tau\cdot 2^j$ and $\ell_j=N (5/4)^j$ so that $a_{j+1} - a_j = a_j$ and $\ell_{j+1} - \ell_j = \ell_j/4$. Define the following events for $j\geq 0$:
    \begin{align*}
        E_j & := \{ D_{a_{j+1}}(v) >1\},\\
        F_j & := \{\# \{i \in (a_j, a_{j+1}] :\, V_i=w\}> m(j+1)\},\\
        G_j & := \{L_{a_{j+1}}(w) < \ell_{j+1}\},\\
        B_j & := E_j \cup F_j \cup G_j,
    \end{align*}
    where $m>0$ is such that $m < N/20$. In words, $E_j$ is the event that vertex $v$ is not a leaf in $T_{a_{j+1}}$. The event $F_j$ corresponds to the event that between steps $(a_j, a_{j+1}]$, new vertices attach to a random neighbour of  $w$ more than $m(j+1)$ times and $G_j$ corresponds to the event that $w$ has less than  $\ell_{j+1}$ leaf neighbours in $T_{a_{j+1}}$. It suffices to prove that %
    \begin{align}\label{eq:bad_events_not_too_likely}
        \p{\bigcup_{j\geq 0}B_j\mid A_N } < 3/4~.
    \end{align}
    Indeed,
    $$\p{v\text{ is an eternal leaf}}=\p{ \cap_{j\geq 0}E_j^c }\geq \p{ \cap_{j\geq 0}B_j^c }~,$$
    since $B_j^c\subset E_j^c$. By \eqref{eq:bad_events_not_too_likely},
    $$\p{ \cap_{j\geq 0}B_j^c }\geq \p{ \cap_{j\geq 0}B_j^c , \ A_N}= \p{A_N}\p{ \cap_{j\geq 0}B_j^c \mid \ A_N}\geq \frac{1}{4}2^{-N}~.$$
    This implies 
    $$\p{v\text{ is an eternal leaf}}\geq 2^{-(N+2)}~, $$
   
    proving Theorem~\ref{thm:StationaryDegree} for $k=1$. To show \eqref{eq:bad_events_not_too_likely} we begin by noting that
    $$ \Cprob{E_j}{\cap_{i=0}^{j-1}B_{i}^c,  A_N} \leq \E{\Cexp{\sum_{i=a_j}^{a_{j+1}} \frac{1}{i\ell_j}}{\tau}}\leq  \frac{1}{\ell_j}(1+\log(2))~,$$
    where the first inequality holds since the conditioning implies that vertex $w$ has degree at least $\ell_j$ from time $a_j$ onwards. Next, since the probability of $V_i=w$ equals to $1/i\leq 1/a_j$ for all $i\geq a_j$, we deduce that, for $j\geq 1$,
    \begin{align*}
        \Cprob{F_j}{\cap_{i=0}^{j-1}B_{i}^c, A_N}  & \leq \p{\text{Bin}(a_{j+1} - a_j, 1/a_j) > m(j+1)} \leq e^{-m(j+1)/3}~,
    \end{align*}
    by a Chernoff bound (\cite[Theorem 2.1.]{JanLucRuc00}). 
    Lastly, under the previous conditioning and given $F_j^c$, the probability of $G_j$ is less than the probability of creating fewer than $\ell_{j+1} +m(j+1) -\ell_j$ new leaf neighbours for $w$ between steps $(a_j, a_{j+1}]$. Conditionally on $F_j^c$, the probability of attaching vertex $i+1$ to $w$ at time $i\in (a_j, a_{j+1}]$ is at least $(\ell_j - m(j+1))/a_{j+1}$. Thus
    \begin{align*}
        \Cprob{G_j\!}{\!\cap_{i=0}^{j-1}B_{i}^c \,,F_j^c,  A_N} & \leq \p{\text{Bin}\!\left(\!a_{j+1} - a_j, \frac{\ell_j - m(j+1)}{a_{j+1}}\right) \leq \ell_{j+1} +m(j+1) -\ell_j}\\
        & \leq e^{-\ell_j/90},
    \end{align*}
    where we use that $m<N/20$, so that $m(j+1)\leq \tfrac{1}{10}\ell_j$ and $\ell_j-m(j+1)>0$ for all $j\geq 0$, and the bound then follows from the Chernoff bound.
    Putting everything together gives us that,
    \begin{align*}
         \Cprob{B_j}{ \cap_{i=0}^{j-1}B_{i}^c , A_N } & \leq \ell_j^{-1}(1+\log(2)) + e^{-m(j+1)/3} + e^{-\ell_j/90},
    \end{align*}
    and 
    \begin{align*}
         \Cprob{\bigcup_{j\geq 0}B_j}{A_N } & \leq \sum_{j\geq 0} \Cprob{B_j}{ \cap_{i=0}^{j-1}B_{i}^c , A_N } \\
         & \leq \frac{1}{N}\sum_{j\geq 0} (4/5)^j(1+\log(2)) + \sum_{j\geq 1}e^{-m(j+1)/3} + \sum_{j\geq 0}e^{-\cN(5/4)^j/90}.
    \end{align*}
   Now, choose $m$ sufficiently large so that the second sum is smaller than $1/4$, then choose $N$ large enough so that $m<N/20$ and the first and the third sum are both smaller than $1/4$; this proves the statement for $k=1$.  
    Next, we adapt the statement for general $k$. Conditionally on $T_n$, if $D_n(v) = k$ then we define $\tau$ as the random time at which a new vertex has been attached to a random friend of vertex $v$ or one of its $k$ neighbours exactly $kN$ times. That is, $\tau = \min\{t: \# \{n< i\leq t: V_i \in \{v\}\cup \cN(v;T_n)\} = kN\}$. Then, with probability at least $(k+1)^{-kN}$ exactly $N$ of these $kN$ new leaves are attached to each of the $k$ neighbours of $v$. We call this event $A_N = \{\forall w \in \cN(v;T_n): \#\{n< i \leq \tau: V_i = w\} = N\}$. Conditionally on $A_N$, we show that $v$ is an eternal degree $k$ vertex with positive probability, because it is likely that all of the neighbours of $v$ continue to acquire many leaf neighbours beyond time $\tau$, making it unlikely for the degree of $v$ to grow. 

   Let $w_1,\dots, w_k$ denote the neighbours of $v$ in $T_n$. Define the events
   \begin{align*}
        E_j^k & := \{D_{a_{j+1}}(v)> k\},\\
        F_j^k & := \bigcup_{l=1}^k\{\#\{ i \in (a_j, a_{j+1}] :\, V_i=w_l\} > m(j+1)\},\\
        G_j^k & := \bigcup_{l=1}^k\{L_{a_{j+1}}(w_l) < \ell_{j+1}\},\\
        B_j^k & := E_j \cup F_j \cup G_j~.
    \end{align*}
    Following the same arguments used in the case of $k=1$, it follows that, choosing $m$ and $N$ sufficiently large, we have
    \[\p{\bigcup_{j\geq 0} B_j^k\mid A_N }<3/4,\]
    which implies the statement for general $k$. 
\end{proof}
Note that it follows from Theorem \ref{thm:StationaryDegree} that, for $E^k_n$ the number of eternal degree $k$ vertices in $T_n$,
\begin{align}\label{eq:pos_frac_forever_deg_k}
    \E{E^k_n} > c_k\,\E{X_n^k}.
\end{align}

\section{Proofs of global properties}\label{sec:proofglobalproperties}

\subsection{Typical distances}[Proof of Theorem~\ref{thm:typical}]
\label{sec:typical}

Theorem~\ref{thm:typical} is a consequence of Theorem \ref{thm:AbundanceHubs}. Indeed, almost surely, one of vertices $1$ and $2$ is a hub, so a positive proportion of vertices is a neighbour of vertex $1$ or of vertex $2$. This implies that with probability bounded away from zero, both $U_n$ and $V_n$ are a neighbour of either vertex $1$ or $2$, so that the distance between them is two and the distance from $U_n$ to $1$ is at most two. More formally, by Theorem~\ref{thm:AbundanceHubs} there is an $i\in\{1,2\}$ and an $\epsilon>0$ so that $\p{Z_i>\epsilon}>\epsilon$. Then, for $n$ large enough, $\p{D_n(i)>\epsilon n/2}>\epsilon/2$. For such $n$, 

\[\p{d_n(U_n,1)\le 2}\ge\p{D_n(i)>\epsilon n/2, d_n(i,U_n)=1}\ge \epsilon^2/4 \]
and 
\[\p{d_n(U_n,V_n)= 2}\ge\p{D_n(i)>\epsilon n/2, d_n(i,U_n)=1, d_n(i,V_n)=1}\ge \epsilon^3/8, \]
which proves the statement.

\subsection{Diameter}[Proof of Theorem~\ref{thm:Diameter}]\label{sec:Diameter}
We show that $\mathrm{Diam}_n$ grows logarithmically almost surely, with explicit asymptotic lower and upper bounds. We start by proving a lower bound.

\begin{lem}\label{lem:lowerbound_diam}
Almost surely 
$$\liminf_n \frac{\mathrm{Diam}_n}{\log(n)}\geq 1~.$$
\end{lem}

\begin{proof}
Among all the paths of length $\mathrm{Diam}_n$ present at time $n$, let us choose one. Denote it by $(i_0\to i_1 \to \cdots \to i_{\mathrm{Diam}_n})$. Let us remark that for $n\geq3$, $\mathrm{Diam}_n$ is always at least $2$. Vertices $i_1$ and $i_{\mathrm{Diam}_n-1}$ are such that at most one of their neighbours is not a leaf (otherwise there would be a path of length $\mathrm{Diam}_n+1$). This implies that, at time $n$, conditioned on $V_n=i_1$, with probability at least $1/2$ we have that $W_n\in \cL(i_1;T_n) $ (the same holds for $i_{\mathrm{Diam}_{n}-1}$). But, if $W_n\in\cL(i_1;T_n)\cup\cL(i_{\mathrm{Diam}_n-1};T_n)$ then the diameter increases by $1$. Because $\Cprob{V_n\in \{ i_1,i_{\mathrm{Diam}_{n}-1} \} }{\mathrm{Diam}_n\geq 3 }=2/n$ and  $\Cprob{V_n\in \{ i_1,i_{\mathrm{Diam}_{n}-1} \} }{ \mathrm{Diam}_n=2 }=1/n$ (note that if $\mathrm{Diam}_n=2$, then $i_1=i_{\mathrm{Diam}_n-1}$), then
\begin{equation}\label{eq:DiameterGrowsFast}
    \Cexp{\Delta \mathrm{Diam}_n }{T_n } \geq \frac{1}{n}\I{\mathrm{Diam}_n\geq 3}+\frac{1}{2n}\I{\mathrm{Diam}_n=2}~. 
\end{equation}
To prove the lemma, we first need to show that, almost surely, $\mathrm{Diam}_n$ reaches $3$ in finite time. Using \eqref{eq:DiameterGrowsFast}, there exists a coupling between $\mathrm{Diam}_n$ and
$$ S_n:=2+\sum_{i=4}^n Z_i~, $$
where $Z_i$ are independent Bernoulli random variables with parameter $1/(2i)$, such that $\mathrm{Diam}_n\geq S_n$ for $n\geq 3$. Let $M$ be the first time when $\mathrm{Diam}_n=3$ and $M'$ the first time when $S_n=3$. By our coupling of $\mathrm{Diam}_n$ and $S_n$, $M\leq M'$. A direct application of \cite[Exercise 2.9]{boucheron:hal-00794821} shows that $\liminf S_n =\infty $ almost surely and therefore $M'$ (and in turn $M$) are finite almost surely. We can now introduce a coupling of $\mathrm{Diam}_n$ from time $M$ onward. Conditionally on $M=m$, \eqref{eq:DiameterGrowsFast} implies that $\mathrm{Diam}_n$ can be coupled from $m$ onward to
$$ H_n^{(m)}:=3+\sum_{i=m}^n X_i~, $$
where $X_i$ are independent Bernoulli random variables with parameters $1/i$, such that $\mathrm{Diam}_n\geq H_n^{(m)}$. Another direct application of \cite[Exercise 2.9]{boucheron:hal-00794821} implies that, for fixed $m$,

$$ \liminf_n \frac{H_n^{(m)}}{\log(n)}\geq 1 \ \text{a.s.} $$
By our coupling, if $M=m$, 
$\liminf \frac{\mathrm{Diam}_n}{\log(n)}\geq 1$ almost surely. Using that $M$ is finite almost surely concludes the proof.\qedhere

\end{proof}

\begin{figure}[h]
\centering
\includegraphics[width=0.7\textwidth]{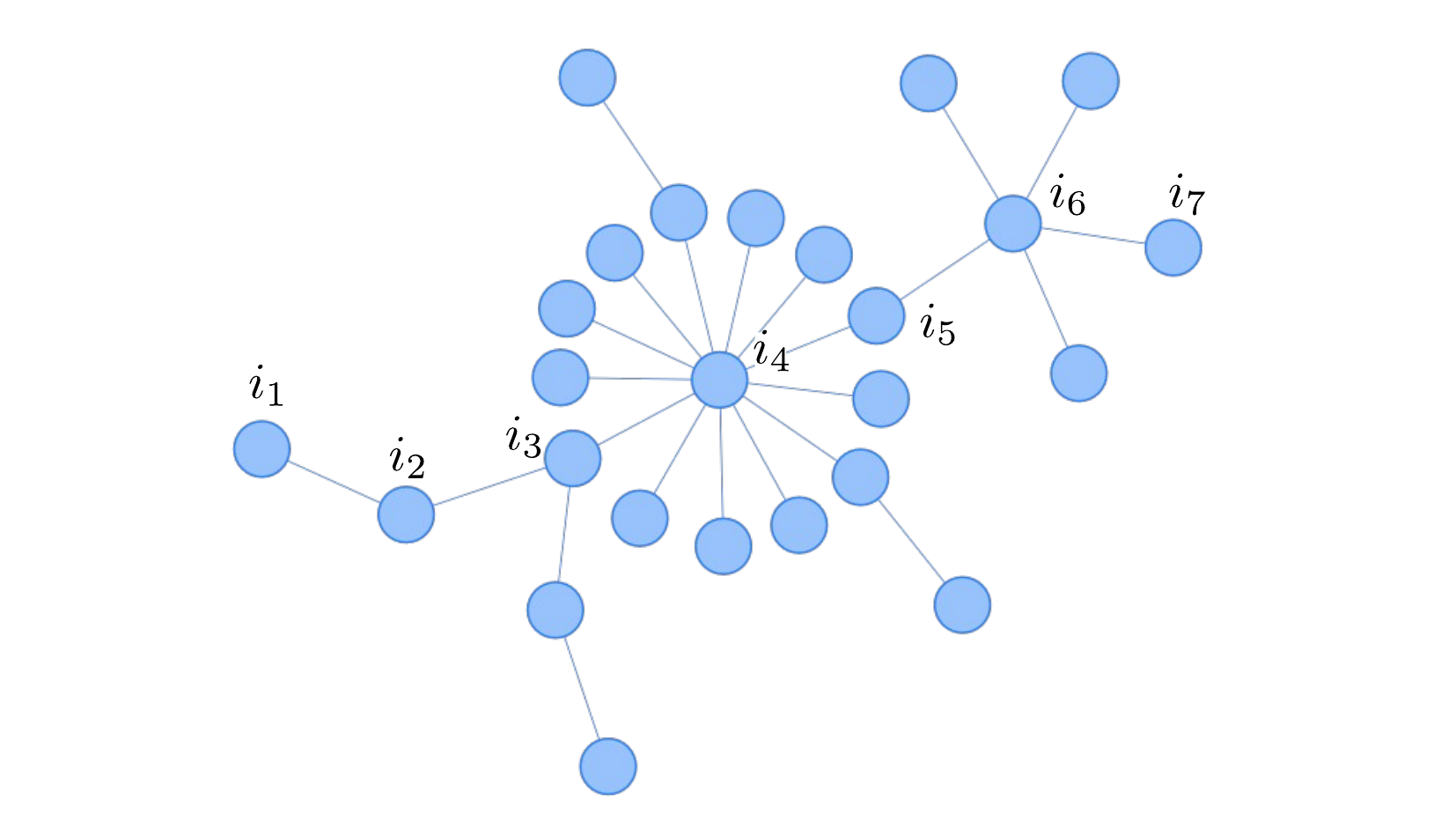}
\caption{Illustration of a RFT of size $25$ and diameter $6$, with vertices of one of the paths of length $6$ highlighted.}
\end{figure}

Next, we prove an upper bound for the diameter.
\begin{lem}\label{lem:lower_bound_diameter}
Almost surely
$$\limsup_n \frac{\mathrm{Diam}_n}{\log(n)}\leq 4e ~.$$
\end{lem}
To prove this lemma we couple the random friend tree with the URRT process. This coupling is the subject of the following lemma. 

\begin{lem}\label{lem:coupling_URRT_1}
The random friend tree $(T_n, n\geq 2)$ and the uniform random recursive tree $(T'_n, n\geq 2)$ can be coupled in such a way that for any $i,j\in [n]$,

\[d_n(i,j)\leq 2 d'_n(i,j)~,\]
where $d_n(i,j)$ is the graph distance between vertex $i$ and $j$ in $T_n$, and $d'_n(i,j)$ is the graph distance between vertex $i$ and $j$ in $T'_n$.
\end{lem}
\begin{proof}
Note that $T_2=T'_2$ so the statement holds for $n=2$. Fix $m\geq 2$ and suppose that we coupled $(T_2,\dots,T_{m})$ and $(T'_2,\dots,T'_{m})$ such that for all $i,j\in [m]$, $d_{m}(i,j)\leq 2 d'_{m}(i,j)$. Now, sample uniformly at random $V_{m}\in [m]$ and let $W_m$ be a uniform neighbour of $V_m$ in $T_{m}$. Let $T_{m+1}$ be the tree obtained by including vertex $m+1$ and edge $\{W_m, m+1\}$ in $T_{m}$ and let $T'_{m+1}$ be the tree obtained by including vertex $m+1$ and edge $\{V_m, m+1\}$ in $T'_{m}$.
Observe that, for $i,j\in [m]$, $d_{m+1}(i,j)=d_{m}(i,j)$ and $d'_{m+1}(i,j)=d'_{m}(i,j)$. Now, let $i\in [m]$ and compute 
\begin{align*}d_{m+1}(i,m+1)&\leq d_{m+1}(i,V_m)+d_{m+1}(V_m,m+1) =d_m(i,V_m)+2\\
&\leq 2 d'_m(i,V_m)+2 \leq 2 d'_m(i,m),\end{align*}
where we use the triangle inequality, the induction hypothesis and the fact that for all $i\in [m]$, $d'_{m}(i,m)\leq 1+d'_{m}(i,V_m)$. 
\end{proof}

\begin{proof}[Proof of Lemma \ref{lem:lower_bound_diameter}] 
Couple $(T_n, n\geq 1)$ to the uniform random recursive tree $(T'_n,n\geq 1)$ as in Lemma \ref{lem:coupling_URRT_1} and observe that 
\begin{align*} \mathrm{Diam}_n&=\max_{i,j\leq n}d_n(i,j)\leq 2 \max_{i,j\leq n}d'_n(i,j)\leq 4 \max_{i\leq n}d'_n(1,i).
\end{align*}
Moreover, by Corollary $1.3$ of \citet{addario2013poisson}, 
\[\frac{\max_{i\leq n}d'_n(1,i)}{\log n}\to e \text{ almost surely,} \]
which concludes the proof of the lemma.
\end{proof}

\subsection{Leaf-depth}[Proof of Theorems \ref{thm:LeafDepth} and \ref{thm:LeafDepthURRT}]\label{sec:Leafdepth}
Theorem \ref{thm:AbundanceHubs} implies that, asymptotically, almost surely each vertex is at distance at most $1$ from a hub. By Theorem \ref{thm:DegreeConvergence}, each hub has mostly leaf neighbours. This suggests that at large times, most vertices that are not leaves are close to a leaf (distance $1$ or $2$). In this section, we show that at large times, there are exceptional vertices that are much further away from the nearest leaf, namely at distance $\Theta(\log n/\log\log n) $. 
We recall that
\[M_n=\max_{i\leq n}\min_{\ell:D_n(\ell)=1}d_n(i,\ell)\]
is the maximal distance of any vertex to the closest leaf at time $n$, which we refer to as the \emph{leaf-depth} at time $n$. 
To prove Theorem~\ref{thm:LeafDepth}, we first need the following lemma to transfer upper bounds on the leaf-depth in the uniform random recursive trees to upper bounds on the leaf-depth in random friend trees. 

\begin{lem}\label{lem:coupling_URRT}
The coupling defined in Lemma \ref{lem:coupling_URRT_1} between the random friend tree $(T_n, n\geq 1)$ and the uniform random recursive tree $(T'_n,n\geq 1)$ also satisfies that for any leaf $\ell'$ in $T'_n$, in $T_n$, the vertex $\ell'$ is at distance at most $1$ from a leaf.
\end{lem}

\begin{proof}
The statement clearly holds for $n\leq 2$. Next, suppose that for some $m$ the statement is satisfied in $T_{m-1}$ and $T'_{m-1}$. Fix an $\ell'\leq m$ so that $\ell'$ is a leaf in $T'_m$. We claim that $\ell'$ is at distance at most $1$ from a leaf in $T_m$. First, if $\ell'=m$, then $\ell'$ is also a leaf in $T_m$ and the claim follows. If $\ell'\leq m-1$, then $\ell'$ is also a leaf in $T'_{m-1}$, so by the induction hypothesis,  $\ell'$ is at distance at most $1$ from a leaf $\ell$ in $T_{m-1}$. If $\ell$ is also a leaf in $T_{m}$, the claim follows. Otherwise, for $\ell$ to be a leaf in $T_{m-1}$, but not $T_m$ it is necessary that $W_m=\ell$. If $W_m=\ell$ and $d_{m-1}(\ell,\ell')=1$, then $\ell'$ is the unique neighbour of $\ell$ in $T_{m-1}$ (because $\ell$ is a leaf in $T_{m-1}$), so $V_m=\ell'$, contradicting that $\ell'$ is a leaf in $T'_m$. Thus, if $W_m=\ell$, then $d_{m-1}(\ell,\ell')=0$, meaning that $\ell=\ell'$. Therefore, in $T_m$, vertex $m$ is a leaf that is at distance $1$ from $\ell'$.\qedhere

\end{proof} 
The next lemma gives an upper bound on the leaf-depth in the uniform random recursive tree. Together with Lemma~\ref{lem:coupling_URRT} we deduce an upper bound for the leaf-depth in random friend trees, given in Proposition \ref{prop:upper_bound_leafdepth}. 

\begin{lem}\label{lem:leafdepth_URRT}
   Let $M'_n$ be the leaf-depth in the uniform random recursive tree at time $n$.  
   Then, for any $\varepsilon>0$,
   \[\p{M'_n\geq (1+\varepsilon)\frac{\log n}{\log\log n}}=o(1).\]
\end{lem}
\begin{proof}
Let $T'_n$ be the uniform random recursive tree at time $n$. For any vertex $v \in T_n'$, we define a canonical path $\cP_n(v)$ in $T'_n$ that ends in a leaf. If $v$ is a leaf in $T'_n$, set $\cP_n(v)=v$. Otherwise, let $w=w_n(v)$ be the neighbour of $v$ in $T'_n$ with the largest label and let $\cP_n(v)$ be $v$ concatenated with $\cP_n(w)$. So, to obtain the path $\cP_n(v)$, start from $v$ and sequentially move to the largest labelled neighbour until reaching a leaf. Define $T'_n(v)$, as the connected component of $v$ in $T'_n$ if the edge between $v$ and its parent was removed (or, equivalently, $T'_n(v)$ is the subtree of $T'_n$ that consists of all vertices that are connected to $v$ by a path on which $v$ is the lowest labelled vertex). Let $|\cP_n(v)|$ be the number of edges on the path. To prove the lemma, it is sufficient to show that $\max_{v\in [n]} |\cP_n(v)|\leq (1+\epsilon)\log n/\log\log n$.\\
First, we check that, for any $\ell\geq 1$, 
\begin{equation}\label{eq:size_youngest_subtree}
\Cprob{|T'_n(w_n(v))|\geq \ell }{ |T'_n(v)|=m}= \begin{cases} \frac{1}{\ell}\text{ if }\ell < m\\ 0 \text{ otherwise,} \end{cases}\end{equation} 
where $m\leq n$. Observe that, conditionally on $|T'_n(v)|=m$, if the vertices in $T'_n(v)$ are assigned labels in $[m]$ that respect the order of the original labels, the resulting tree has the same law as $T'_{m}$. Therefore, \eqref{eq:size_youngest_subtree} follows if, for any $m$, we have that for all $\ell \geq 1$, 
\begin{equation*}
\p{|T'_m(w_m(1))|\geq \ell}= \begin{cases} \frac{1}{\ell}\text{ if }\ell \leq m\\ 0 \text{ otherwise,} \end{cases}\end{equation*} 
where we recall that $T'_m(w_m(1))$ is the subtree rooted at the youngest child of $1$. 

The proof is by induction on $m$. The statement clearly holds for $m=1$. Now, suppose the statement holds for $m=k-1$. For $m=k$, the statement is obvious for $\ell=1$ and $\ell>k$ since $1\leq |T'_k(w_k(1))| \leq k$. 
Observe that, if $V_{k}=1$ (i.e, if vertex $k$ connects to vertex $1$), then $T'_{k}(w_{k}(1))$ consists only of the vertex $k$ in which case $|T'_{k}(w_{k}(1))|=1$. If $V_k\in  T'_{k-1}(w_{k-1}(1))$, then $T'_{k}(w_{k}(1))$ is composed of the vertices of $T'_{k-1}(w_{k-1}(1))$ and vertex $k$, so $|T'_{k}(w_{k}(1))|=|T'_{k-1}(w_{k-1}(1))|+1$. If $V_k \notin \{1\}\cup T'_{k-1}(w_{k-1}(1))$ then $T'_{k}(w_{k}(1))=T'_{k-1}(w_{k-1}(1))$, giving $|T'_{k}(w_{k}(1))|=|T'_{k-1}(w_{k-1}(1))|$. Therefore, for $1<\ell\leq k $,
\begin{align*}
    \set{|T'_k(w_k(1))|\geq \ell}=& \set{\set{|T'_{k-1}(w_{k-1}(1))|\geq \ell}\cap \set{V_k\neq 1 }}\\&\cup \set{\set{|T'_{k-1}(w_{k-1}(1))|=\ell-1}\cap \set{V_k\in T'_{k-1}(w_{k-1}(1)) }}.
\end{align*} 
By the induction hypothesis, for $\ell < k$, 
\begin{align*}
    \p{|T'_k(w_k(1))|\geq \ell}
    =\frac{1}{\ell}\frac{k-1 }{k}+\frac{1}{\ell(\ell-1)}\frac{\ell-1}{k}=\frac{1}{\ell},
\end{align*}
and for $\ell=k$,
\begin{align*}
    \p{|T'_k(w_k(1))|\geq k }
    =0+\frac{1}{k-1}\frac{k-1}{k}=\frac{1}{k}.
\end{align*}
The equality \eqref{eq:size_youngest_subtree} follows for all $m$. 
Now, define $w^{(1)}=w_n(v)$ the vertex at distance $1$ from $v$ on $\cP_n(v)$ and $w^{(\ell)}=w_n(w^{(\ell-1)})$ the vertex at distance $\ell$ from $v$ on  $\cP_n(v)$. Then,
\[ \left\{ |\cP_n(v)|\geq k \right\}=\{|T'_n(w^{(1)})|\geq k\}\cap \{|T'_n(w^{(2)})|\geq k-1\}\cap \dots \cap \{|T'_n(w^{(k)})|\geq  1\}.\]
Together with \eqref{eq:size_youngest_subtree}, this implies that 
\[\p{|\cP_n(v)|\geq k}\leq\frac{1}{k!}\leq \frac{e^k}{k^k} .\]
Then, fix $\varepsilon>0$ and subsitute $(1+\varepsilon)\frac{\log n}{ \log\log n}$ to $k$. The above equation directly implies that
\[\p{|\cP_n(v)|\geq (1+\varepsilon)\frac{\log n}{ \log\log n}}=o(n^{-1}) .\]
Finally, a union bound implies Lemma~\ref{lem:leafdepth_URRT}.\qedhere
\end{proof}

From Lemmas~\ref{lem:coupling_URRT}and~\ref{lem:leafdepth_URRT} we obtain an upper bound on the leaf depth in random friend trees.
\begin{prop}\label{prop:upper_bound_leafdepth}
For any $\varepsilon>0$
   \[\p{M_n\geq (2+\varepsilon)\frac{\log n}{\log\log n}}=o(1).\]
\end{prop}

\begin{proof}
Couple the random friend tree $(T_n, n\geq 1)$ and the uniform random recursive tree $(T'_n,n\geq 1)$ as in Lemma~\ref{lem:coupling_URRT}. Then, fix $i\leq n$ such that for $N'_n(i)$, the degree of vertex $i$ in $T'_n$, we see that 
\begin{align*}
    \min_{\ell:D_n(\ell)=1} d_n(i,\ell) & {\leq \min_{\ell:D_n(\ell)=1}\min_{\ell':N'_n(\ell')=1} \left(d_n(i,\ell')+  d_n(\ell', \ell)\right)}\\
    & \leq \min_{\ell':N'_n(\ell')=1} d_n(i,\ell')+1 \\
    &\leq \min_{\ell':N'_n(\ell')=1}2d'_n(i,\ell')+1,
\end{align*}
where the last two inequalities follow from the properties of the coupling. By taking the maximum over $i\in [n]$, we have $M_n \leq 2 M_n'+1$ and so Lemma \ref{lem:leafdepth_URRT} implies the proposition. \qedhere
\end{proof}

To conclude the proof of Theorem \ref{thm:LeafDepth} it remains to prove an asymptotic lower bound for $M_n$. In order to show that the leaf depth is at least of order $\log n/ \log \log n$ we first present a proof of the corresponding result for the URRT. To the best of our knowledge, this result does not appear elsewhere. Moreover, the proof is less technical but has the same structure as its counterpart for random friend trees, so it is a good way to introduce the ideas needed for our main proof. In doing so, we hope that the technicalities in the proof of Lemma \ref{prop:leafdepth} are easier to understand.
\begin{prop}\label{prop:leafdepthURRT}
    Let $M_n'$ be the leaf-depth in the uniform random recursive tree at time $n$. Then, $M'_n=\Omega_p(\log n/\log\log n)$.
\end{prop}

\begin{proof}
Let $T_n'$ be the uniform random recursive tree at time $n$. In a uniform random recursive tree, the number of vertices with degree at least $3$ goes to infinity almost surely (see \citet{janson2005asymptotic}). We call such a vertex a branch point. Therefore, we can choose $n$ sufficiently large such that there is at least one branch point. Let $P'_n$ be the maximal distance of any leaf in $T_n'$ to the nearest branch point, that is,
\[P'_n=\max_{\ell: N'_n(\ell)=1}\min_{j:N'_n(j)\geq3} d'_n(\ell,j).\]
We remark that $M'_n\geq P'_n/2$ because a midpoint of a longest leaf-to-branchpoint path is at distance at least $P'_n/2$ from the nearest leaf.  It is well known (see \citet{janson2005asymptotic}) that the proportion of leaves in an URRT tends to $1/2$, consequently $X'^{\geq 2}_m/m\to 1/2$ almost surely. Thus, for $\eps>0$ and $n$ sufficiently large, 
$$\p{1-\frac{X'^{\geq 2}_k}{k} \leq \frac{1}{3}; \  \forall k \geq n }> 1 - \eps.$$
Now, condition on the number of leaves at time $n/2$ being at least $n/6$, that is $X'^{1}_{n/2} \geq n/6$, and let $\{v_1,\cdots,v_{P}\}$ be an arbitrary set of $P = n/6$ leaves of $T_{n/2}'$. We study the subtrees rooted at $v_i$ and will show that at time $n$, with high probability, for at least one $i\in [P]$ the subtree of $v_i$ contains a path from $v_i$ ending in a leaf and solely consisting of $\Omega(\log n/\log\log n)$ vertices of degree two. To this end, we say that a path consisting of degree two vertices that ends in a leaf \emph{grows} at time $m$ if vertex $m+1$ attaches to the leaf at the end of the path. This increases the length of the path by $1$. We say that a path consisting of degree two vertices that ends in a leaf \emph{dies} at time $m$ if vertex $m+1$ connects to a degree-two vertex on the path that is at distance at most $K$ from the leaf. That is, we only keep track of paths up to distance $K$ from the leaf. Note that at each time step only one path can grow or die. We will only track paths until their first death. Note that, at time $m$, for $i\in [P]$, conditionally on the path rooted at $v_i$ has not died yet, the probability of the path growing is $1/m$ and the probability of the path dying is at most $K/m$. Observe that $P_n'$ stochastically dominates the minimum of $K$ and the length of the longest path at time $n$ rooted at some $v_i$ for $i\in [P]$ that has not died. We can then couple these path growth processes to a balls-in-bins model as follows. Let $P$ be the number of bins. The process is started at time $n/2$ with $P$ empty bins. For each time $m\in [n/2,n]$, with probability $P\cdot K/m$ add a black ball to a uniform random bin, or with probability $P/m$ add a white ball to a uniform random bin. By doing so, a black ball is added to a given bin with probability $K/m$ and a white ball is added to a given bin with probability $1/m$. We further see that $P_n'$ stochastically dominates the smallest number between K and the maximum number of white balls at time $n$ in a bin with zero black balls. It therefore suffices to show that for some $c>0$ and $K:= c\log n/\log \log n$, for $n$ large enough, with high probability, one of the $P$ bins contains at least $K$ white balls and no black balls. Observe that, for $n$ sufficiently large, between times $n/2$ and $n$, with probability at least $1-\epsilon$, at most $B=2PK$ black balls and at least $W = P/8$ white balls are added. Conditioned on this event, the probability that a specific bin contains at least $K$ white balls is at least 
\begin{align*}
    {W \choose{K}} P^{-K}(1-P^{-1})^{W-K} > e^{-1}\left(\frac{W}{PK}\right)^K = e^{-1}\left(\frac{1}{8K}\right)^K,
\end{align*}
for large $n$, where we bound $(1-P^{-1})^{W-K} \geq (1-P^{-1})^{W} = ((1-6/n)^{n/6})^{1/8} > e^{-1}$. Therefore the expected number of bins containing at least $K$ white balls can be bounded from below by 
$$P\cdot e^{-1}(8K)^{-K}= e^{-1} \frac{n}{6}\left(\frac{\log \log n}{8 c\log n}\right)^{c\log n/\log \log n}> n^{1-c}$$
for large $n$.

Then, if $c<1/2$, since the numbers of white balls in two distinct bins have negative covariance, Chebyshev's inequality bounds the probability that the number of paths with at least  $K$ growth events is less than $n^{1-2c}$ tends to $0$, so in particular is smaller than $\epsilon$ for $n$ sufficiently large. Any of these has $0$ black balls with probability at least $(1-P^{-1})^{B}\geq e^{-3K}=\omega(n^{1-2c})$, so another straightforward application of the second moment method implies that for $n$ large enough, for any $c<1/2$, with probability at least $1-3\epsilon$ at time $n$, there is a bin containing $c\log(n)/\log\log(n)$ white balls and no black balls. Recalling the stochastic dominance, that is valid on an event of probability at least $1-\epsilon$, with probability at least $1-4\epsilon$ there is a leaf at distance at least $c \log n/\log\log n$ from the nearest branch point. \qedhere
\end{proof}

Using a similar proof, we can prove an asymptotic lower bound for the leaf-depth in random friend trees. 
\begin{prop}\label{prop:leafdepth}
     Let $M_n$ be the leaf-depth in the random friend tree at time $n$. Then,
    $M_n=\Omega(\log n/\log\log n)$ in probability.
\end{prop}
\begin{proof}

By Theorem~\ref{thm:SmallDegreeVertices}, $X_n^{\geq 3}\to \infty$ almost surely, so we may pick $n$ large enough such that $X_n^{\geq 3}\geq 1$. Then, define 
\[P_n=\max_{\ell: D_n(\ell)=1}\min_{j:D_n(j)\geq3} d_n(\ell,j)\] 
to be the largest distance from a leaf to the nearest branch point in $T_n$. We see that $M_n\geq P_n/2 $ because the midpoint of the longest leaf-to-branchpoint path is at distance at least $P_n/2$ from the nearest leaf. Therefore, it suffices to show that $P_n=\Omega(\log n/\log\log n)$ in probability. 

As in Proposition \ref{prop:leafdepthURRT} the proof is divided into two steps. We first show that, for $n$ large enough, with high probability at time $n/2$ there are at least $n^{\delta/2}$ leaves attached to a degree two vertex. Remark that this step requires slightly more work than the URRT case. For the URRT, one only needs to ensure the presence of many leaves at time $n/2$ to guarantee that that a long path of subsequent degree-two vertex will emerge from one of these leaves. For the random friend tree, one needs to ensure that, at time $n/2$, there exists many leaves attached to a degree two vertex,
 to guarantee that new vertices attach to these leaves with some sufficient probability. This difference is due to the attachment rule in random friend trees. Denote by $\cL_{n}$ the set of leaves attached to a vertex of degree two at time $n$ and let $H_n=|\cL_{n}|$ be the number of leaves attached to a degree two vertex in $T_n$. In the second part of the proof, we show that it is likely that at time $n$, a path of subsequent degree two vertices longer than $c\log n/\log\log n $ grows from at least one of the leaves of $\cL_{n/2}$, that is,
\[\max_{ \ell \in \cL_{n}}\min_{j:D_n(j)\geq3} d_n(\ell,j)>c\frac{\log n}{\log\log n}~,\]
with high probability. Lemma \ref{lem:atleast2_vs_atleast3} states that 
\[ \Cexp{\Delta X_n^{\geq 2}}{T_n}\geq \frac{1}{3n}X_n^{\geq 3}~,\]
and by Theorem~\ref{thm:SmallDegreeVertices}, $n^{-\delta}X_n^{\geq 3}\to \infty$ almost surely, for $0.1 < \delta < 0.9$. Therefore, for fixed $\varepsilon>0$ and sufficiently large $n$,
$$\p{\frac{1}{3m} X_m^{\geq 3}>m^{-(1-\delta)}\ \forall m\geq n/4}\geq 1-\epsilon~.$$
It follows that, for all $n/4\leq m \leq n/2$,
\[\Cprob{\Delta X_m^{\geq 2}=1}{  \frac{1}{3m} X_m^{\geq 3}>m^{-(1-\delta)}}\geq  n^{-(1-\delta)}\]
because $m^{-(1-\delta)}\geq n^{-(1-\delta)}$ if $m\leq n/2$.  

We will show that $H_n$ grows polynomially in probability. Note that for every vertex $v \in \cL_m$, $\p{W_m=v}=1/(2m)$ and $\Delta H_n =1 $ if and only if $W_m$ is a leaf in $T_m$ that is not in $\cL_m$. Further, $\Delta X^{\ge 2}_m = 1$ if and only if $W_m$ is a leaf in $T_m$. By summing over all possible values of $W_m$ we therefore get the lower bound 
\begin{align*}
    & \Cprob{\Delta H_m=1}{ H_m=k,\ \frac{1}{3m} X_m^{\geq 3}>m^{\delta-1}}\\
    &\ge  \Cprob{\Delta X_m^{\geq 2}=1}{ \frac{1}{3m} X_m^{\geq 3}>m^{-(1-\delta)}}-\frac{k}{2m}\ge n^{\delta-1}-\frac{2k}{n}~,
\end{align*}
for $n/4\le m \le n/2$. Therefore, either $H_m \geq n^{\delta} /16$ or $\Cprob{\Delta H_m = 1}{ \frac{1}{3m} X_m^{\geq 3}>m^{-(1-\delta)}}$ $\geq \frac{7}{8}n^{-(1-\delta)}$ for $n/4\le m \le n/2$.

Finally, $\Delta H_m=-1$ if $W_m$ is a degree two vertex attached to a leaf, and there are exactly $H_m$ such vertices in $T_m$. For $v$ a given vertex of degree two attached to a leaf, $\p{W_m=v}\leq 2/m\leq 8/n$ for $n/4\leq m \leq n/2$. By a union bound, for $n/4\leq m\leq n/2$, 
$$\Cprob{\Delta H_m=-1}{ H_m=k}\leq 8 k/n~.$$ 
By the two arguments above, for $n/4\leq m\leq n/2$,
$$\p{\Delta H_m=-1\mid H_m\leq \frac{1}{16}n^{\delta}}\leq \frac{1}{2}n^{-(1-\delta)}~,$$
and
$$\p{\Delta H_m=1\mid H_m\leq \frac{1}{16}n^{\delta}, \ \frac{1}{3m} X_m^{\geq 3}>m^{\delta-1}}\geq \frac{7}{8}n^{-(1-\delta)}~,$$
Then, it follows that, on the event $\left\{ \forall m \in [n/4,n/2], \ \tfrac{1}{3m} X_m^{\geq 3}>m^{\delta-1}\right\}$, with probability at least $1-\varepsilon$, for $n$ sufficiently large, $H_m$ grows to at least $\frac{1}{17} n^{\delta}$ between times $n/4$ and $n/2$. Finally, for $n$ large enough, $ n^{\delta}/17>n^{\delta/2}$, so
 $$\p{H_{n/2}\geq n^{\delta/2}}\geq 1-2\epsilon~,$$
which concludes the first part of the proof.

We now condition on the event $\{H_{n/2}\geq n^{\delta/2}\}$. We show that, on this event, it is likely that at time $n$ at least one of the leaves of $\cL_{n/2}$ is at distance at least $\log n/\log\log n$ from the nearest branch point. The second part of the proof is identical from the URRT proof, but we present it again for clarity.

Let $\{v_1,\ldots,v_{P}\}$ be an arbitrary set of $P=n^{\delta/2}$ leaves in $\cL_{n/2}$. We study the subtrees rooted at $v_i$ and will show that at time $n$, with high probability, for at least one $i\in [P]$ the subtree of $v_i$ contains a path ending in a leaf and solely consisting of $\Omega(\log n/\log\log n)$ vertices of degree two. To this end, we say that a path consisting of degree two vertices that ends in a leaf \emph{grows} at time $m$ if vertex $m+1$ attaches to the leaf at the end of the path. This increases the length of the path by $1$. We say that a path consisting of degree two vertices that ends in a leaf \emph{dies} at time $m$ if vertex $m+1$ connects to a degree two vertex on the path that is at distance at most $K$ from the leaf. That is, we only keep track of paths of length at most $K$. Note that, at each time step, only one path can grow or die. We will only track paths until their first death. Also note that, at time $m\in[n/2,n]$, for $i\in [P]$, conditionally on the path rooted at $v_i$ has not died yet, the probability of the path growing is $1/(2m)$ and the probability of the path dying is at most $K/m$. Observe that $P_n$ stochastically dominates the minimum of $K$ and the longest path at time $n$ rooted at some $v_i$ for $i\in [P]$ that has not died. We can then couple these path growth processes to a balls-in-bins model. Let $P$ be the number of bins and start the process with $P$ empty bins. For each time $m\in [n/2,n]$, with probability $P\cdot K/m$ add a black ball to a uniform random bin, or with probability $P/(2m)$ add a white ball to a uniform random bin. By doing so, a black ball is added to a given bin with probability $K/m$ and a white ball is added to a given bin with probability $1/(2m)$. We further see that $P_n$ stochastically dominates the smallest number between $K$ and the maximum number of white balls at time $n$ in a bin with zero black balls. It therefore suffices to show that, for some $c>0$ and $K:= c\log n/\log \log n$, with high probability at least one of the $P$ bins contains at least $K$ white balls and no black balls. Observe that, for $n$ sufficiently large, between time $n/2$ and $n$, with probability at least $1-\epsilon$, at most $B=2PK$ black balls and at least $W = P/8$ white balls are added. Conditioned on this event, the probability that a specific bin has at least $K$ white balls is at least

\begin{align*}
    {W \choose{K}} P^{-K}(1-P^{-1})^{W-K} > e^{-1}\left(\frac{W}{PK}\right)^K = e^{-1}\left(\frac{1}{8K}\right)^K,
\end{align*}
for large $n$, where we bound $(1-P^{-1})^{W-K} \geq (1-P^{-1})^{W} = ((1-n^{-\delta/2})^{n^{-\delta/2}})^{1/8} > e^{-1}$. The expected number of bins containing at least $K$ white balls is bounded from below by 

$$P\cdot e^{-1}(8K)^{-K}= e^{-1} n^{\delta/2}\left(\frac{\log \log n}{c\log n}\right)^{c\log n/\log \log n}> n^{\delta/2-c}~,$$
for large $n$.

Then, if $c<\delta/8$, since the numbers of white balls in two distinct bins have negative covariance, Chebyshev's associaton inequality gives that the probability that the number of paths with at least  $K$ growth events is less than $n^{\delta/2-2c}$ tends to $0$, so in particular, is smaller than $\epsilon$ for $n$ sufficiently large. Any of these contains no black balls with probability at least $(1-P^{-1})^B\geq e^{-3K}=\omega (n^{\delta/2-2c})$, so another straightforward application of the second moment method implies that for $n$ large enough, for any $c<\delta/8$, with probability at least $1-3\epsilon$ there is a bin with $c\log(n)/\log\log(n)$ white balls and no black balls at time $n$. Recalling the statistical dominance, that is valid on an event of probability at least $1-2\epsilon$, with probability at least $1-5\epsilon$, there is a leaf at distance at least $c \log n/\log\log n$ from the nearest branch point. \qedhere
\end{proof}

\subsection{High-degree vertices}[Proof of Theorems~\ref{thm:numberofhubs} and \ref{thm:InfinityLinearDegree}]\label{sec:LargeDegreeVertices}

Recall that $Z_v=\liminf_{n\to\infty}\frac{D_n(v)}{n}$, and that by Theorem~\ref{thm:DegreeConvergence}, in fact 
\[Z_v=\lim_{n\to\infty}\frac{D_n(v)}{n}\text{ almost surely.}\]

The following lemma, combined with Theorem~\ref{thm:SmallDegreeVertices}, gives a lower bound on the number of hubs and proves Theorem~\ref{thm:numberofhubs}.
\begin{lem}\label{lem:number_vertices_infinity} Almost surely,
    \[\#\{v\in [n]: Z_v>0 \}>\tfrac{1}{2}X^{\geq2}_n. \]
\end{lem}

We prove this lemma using Lemma \ref{lem:edge_cover_half_non_leaves}, below, but we need some additional definitions for its statement. For a graph $G=(V,E)$, we say $V'\subset V$ is an \emph{edge cover} of $G$ if for each $e\in E$, there is a $v\in V'$ such that $v\in e$. 
Define the \textit{minimal edge cover number} of a graph $G=(V,E)$, denoted by $EC(G)$, as follows
\begin{align}\label{def:min_edge_cover}
    EC(G) := \min\{|V'|: V'\text{ is an edge cover of  }G\}.
\end{align}
Lemma~\ref{lem:number_vertices_infinity} is a direct consequence of Theorem \ref{thm:AbundanceHubs}, which states that each edge contains a hub, and the following lemma.
\begin{lem}\label{lem:edge_cover_half_non_leaves}
    For any tree $t$, we have that 
    \[|EC(t)| \geq \tfrac{1}{2}X^{\geq 2}(t),\]
    where $ X^{\geq 2}(t)$ denotes the number of non-leaves in the tree $t$.
\end{lem}
\begin{proof}
    We can assume that $t$ is a rooted tree by declaring an arbitrary vertex in $t$ the root. Decompose the tree $t$ into the following vertex-disjoint paths. Let $\ell_1,...,\ell_{X^1(t)}$ be the leaves of $t$. For a leaf $\ell$ of $t$, let $P(\ell)$ be the path from $\ell$ to the root of $t$. For each $i\in [X^1(t)]$, let $P'(\ell_i) = P(\ell_i)\setminus \bigcup_{j=0}^{i-1}P'(\ell_j)$, that is, $P'(\ell_i)$ is the path $P(\ell_i)$ stripped of the vertices in $\bigcup_{j=0}^{i-1}P'(\ell_j)$. This decomposition gives us $X^{1}(t)$ disjoint paths $P'_{\ell_1}, \ldots, P'_{\ell_{X^1(t)}}$. Note that if $V'\subset V$ is an edge cover of $t$ it must also be an edge cover of $P'(\ell_1)\cup \ldots \cup P'(\ell_{X^1(t)})$ (indeed, while removing edges, the requirement for a collection of edges to be an edge cover is weakened). An edge cover of a disconnected graph is a disjoint union of edge covers of the components, and an edge cover of a path of $m$ vertices contains at least $\lfloor m/2 \rfloor$ vertices, so 
    \[
        |EC(t)|  \geq \sum_{i=1}^{X^1(t)}\left\lfloor\frac{|P'_i|}{2}\right\rfloor \geq \sum_{i=1}^{X^1(t)}\left(\frac{|P'_i| -1}{2}\right) = \frac{|V(t)| - X^1(t)}{2} = \frac{X^{\geq 2}(t)}{2}. \qedhere
    \]
\end{proof}

We now prove Theorem \ref{thm:InfinityLinearDegree}, which in particular implies that for any $k$, the number of vertices with degree at least $k$ goes to infinity almost surely. 

\begin{proof}[Proof of Theorem \ref{thm:InfinityLinearDegree}]
Fix a constant $M\in \N$  and let $(m_n,n\geq 1)$ be a sequence satisfying $m_n=o(n)$. We will prove that $\liminf_{n\to\infty} X^{\geq m_n}_n \geq M$ almost surely. By Lemma \ref{lem:lowerbound_diam}, there exists an almost surely finite time $\tau$ such that the diameter of the tree at time $\tau$ exceeds $2M$. Fix an arbitrary path of $2M+1$ vertices in $T_\tau$ and, for $n\geq \tau$, let $N^{(1)}_n\geq \dots\geq N^{(2M+1)}_n$ be the degrees of the vertices on this path in decreasing order. 
Then, by Theorem \ref{thm:AbundanceHubs}, almost surely, at least $M$ vertices on this path are hubs, so 
\[\liminf_{n\to \infty} \frac{N^{(M)}_n}{n}>0\] and in particular, there is a finite time $\tau'$ such that $N^{(M)}_n>m_n$ for all $n\geq \tau'$.  
This implies that from time $\tau'$ onwards, there are at least $M$ vertices with degree at least $m_n$, so 
\[\p{\liminf_{n\to\infty} X^{\geq m_n}_n \geq M}=1.\qedhere\]
\end{proof}

\subsection{Low-degree vertices}[Proof of Theorems~\ref{thm:SmallDegreeVertices} and \ref{thm:all_finite_deg_large}]\label{sec:SmallDegreeVertices} 
In this subsection we prove polynomial upper and lower bounds of $X_n^{\geq k}$, for $k$ bounded and show that $X_n^{\geq k} = \Theta(X_n^{\geq 2})$ almost surely. We begin by stating a general result on adapted processes, which will be of use in the proof of Theorem~\ref{thm:SmallDegreeVertices}. Its proof can be found in the appendix. For each $k \in \N$, recall that $X^k_n$ is the number of vertices of degree $k$ in $T_n$ and $X^{\geq k}_n$ is the number of vertices of degree at least $k$ in $T_n$.

\begin{prop}\label{prop:`lines'}
Let $(X_k)_{k\geq 0}$ and $(Y_k)_{k\geq 0}$ be integer-valued non-decreasing processes adapted to some filtration $(\cF_k)_{k\geq 0}$ such that $0\leq \Delta X_k + \Delta Y_k \leq 1$ for all $k$. Suppose there exists $\alpha>0$ such that $\Cexp{\Delta X_k }{ \cF_k}\geq  \alpha \Cexp{\Delta Y_k }{\cF_k}$ on the event $\{X_k < \alpha Y_k\}$, except at finitely many times almost surely. Then for any $\beta \in (0,\alpha)$ there exists $C>0$ such that 

$$\p{X_n < \beta Y_n - C\log n \text{ infinitely often}} = 0~,$$ 
and for all $n$, $\E{X_n} \geq \beta \E{Y_n} - C\log n$.

\end{prop}

\begin{lem}
For $n \ge 4$,
\begin{align}
    \Cexp{\Delta X_{n}^{\geq 2} }{ T_n } &\leq \frac{1}{2n}X_n^{\geq 2}+\frac{1}{2n}X_n^{\ge 3}~, \label{eq:diff1}
    \end{align}
    and for $n \ge 5$,
    \begin{align}
    \Cexp{\Delta X_n^{\ge 3}}{ T_n}\le \frac{4}{3n}X_n^2~.
     \label{eq:diff2}
\end{align}
\end{lem}
\begin{proof}
To prove \eqref{eq:diff1}, note that 
$\Delta X_{n}^{\geq 2}>0$ precisely if $W_n$ is a leaf of $T_n$, 
so 
\[
\Cexp{\Delta X_{n}^{\geq 2} }{T_n}
= 
\frac{1}{n}\sum_{v\in T_n}\frac{L_{n}(v)}{D_{n}(v)}.
\]
When $n \ge 3$, leaves do not have leaves as neighbours, and when $n \ge 4$, any vertex $v$ of degree two in $T_n$ has at most one leaf neighbour, thus if $v$ has degree two, $L_n(v)/D_n(v)=L_n(v)/2\le 1/2$. Together with the previous equality, this implies that 

\[
\Cexp{\Delta X_{n}^{\geq 2}}{ T_n}
\le 
\frac{1}{n}\left(\sum_{\{v:D_n(v)=2\}}\frac{1}{2} + 
\sum_{\{v:D_{n}(v)\geq 3\}}\frac{L_{n}(v)}{D_{n}(v)}\right)
\le \frac{X_n^2}{2n}+\frac{X_n^{\ge 3}}{n}~.
\]
Since $X_n^{2}=X_n^{\ge 2}-X_n^{\ge 3}$, this implies \eqref{eq:diff1}.

\vspace{0.5cm}

To prove \eqref{eq:diff2}, for $1 \le i \le j$ let 
\[
S_{ij}=\{v \in T_n: D_n(v)=2,\mbox{ the neighbours of $v$ in $T_n$ have degrees $i$ and $j$}\}\,.
\]
Then $X_n^2=\sum_{1 \le i \le j} |S_{i,j}|$ and 
\[
\Cexp{\Delta X_{n}^{\geq 3} }{T_n}
=
\sum_{1 \le i \le j}
\Cprob{W_n \in S_{ij}}{T_n}
=
\frac{1}{n} \sum_{1 \le i \le j} |S_{ij}| \left(\frac1i +\frac1j\right)\, .
\]
We bound this sum by splitting it into three sums. First, for terms with $i=1$ and $j \ge 3$ we have $1/i+1/j \le 4/3$. For $v \in \bigcup_{2 \le i \le j} S_{ij}$ we have $1/i+1/j \le 1$, while for $v \in S_{12}$ we have $1/i+1/j=3/2$. 
We claim that at most half of the vertices in $S_{12}\cup\bigcup_{2 \le i \le j} S_{ij}$ can be in $S_{12}$. Indeed, provided that $n \ge 5$, if $v\in S_{12}$ then for $u$ its unique neighbour with degree $2$ it holds that $u\in \bigcup_{2 \le i \le j} S_{ij}$.  Moreover, $n\ge 5$ implies that $u$ has at most one neighbour in $S_{12}$. Therefore, $|S_{12}| \leq |\bigcup_{2 \le i \le j} S_{ij}| $. This implies that 
\begin{align*}
\sum_{1 \le i \le j} |S_{ij}| \left(\frac1i +\frac1j\right)
& \le  \frac{4}{3}\sum_{j \ge 3} |S_{1j}| + \frac{1}{2} \left(1+\frac32\right) \left(|S_{12}|+\sum_{2 \le i \le j} |S_{ij}|\right)
\le \frac{4}{3}X_n^2\, ,
\end{align*}
therefore $\Cexp{\Delta X_{n}^{\geq 3} }{T_n } \le \frac{4}{3n} X_n^2$ as claimed.
\end{proof}

Before stating the next lemma we introduce the notation $X^{ k,\leq k}_n$, the number of vertices of degree $k$ having at most one neighbour of degree at least $k+1$, and $X^{ k,>k}_n$, the number of vertices of degree $k$ with at least two neighbours of degree at least $k+1$.
\begin{lem}\label{lem:fastgrowthverticesdegreek}
For any positive integer $k$, $n\geq 3$, 
\begin{align}
    X^{k}_n&=X^{ k,>k}_n+X^{k,\leq k }_n,\label{eq:diff3}\\
X^{\geq k+1}_n&\geq X^{k,>k}_n,\label{eq:diff4}\\
\Cexp{\Delta X^{\geq k+1}_{n}}{ T_n}&\geq \frac{k-1}{kn}X^{k,\leq k}_n\label{eq:diff5}.
\end{align}
\end{lem}
\begin{proof}
The equality \eqref{eq:diff3} follows directly from the definition of $X^{ k,>k}_n$ and $X^{k,\leq k }_n$. To prove the second statement, remark that  any vertex $v$ contributing to $X_n^{k,>k}$ has at least two neighbours of degree at least $k+1$, and at least one is a child of $v$. Since every vertex is the child of at most $1$ vertex, \eqref{eq:diff4} follows.

Finally, to prove \eqref{eq:diff5}, one must understand how vertices of degree $k+1$ are created. In order to have $\Delta X_n^{\ge k+1} =1$, it is sufficient (but not necessary) that vertex $n+1$ attaches to a vertex counted by $X_n^{k,\le k}$, or in other words, that $W_n=w$ for some vertex $w$ with degree $k$ which has at least $k-1$ neighbours of degree at most $k$. For each such vertex $w$, this happens with probability at least $\tfrac{k-1}{kn}$. This proves \eqref{eq:diff4}.
\end{proof}

\begin{lem}\label{lem:upperbound_change_atleastk}
For any integer $k$, whenever $n\ge k+2$.
\begin{align}\label{eq:upperbound_change}
\Cexp{\Delta X^{\geq k+1}_{n}}{T_n}\leq \frac{k-1/2}{n}  X^{k}_n~.
\end{align}
\end{lem}
\begin{proof}
A vertex of degree $k+1$ is created at time $n$ if $W_n$ has degree $k$ in $T_n$. For a vertex $w\in T_n$ of degree $k$, the probability of $W_n = w$ is maximized if the neighbours of $w$ have lowest possible degree; that is, if $w$ has $k-1$ leaf neighbours and one neighbour of degree two. In this case, the probability that $W_n = w$ equals $\tfrac{k-1/2}{n}$; the lemma follows.
\end{proof}

We use the following two lemmas to show that the number of non-leaves grows at least polynomially.
\begin{lem}\label{lem:atleast2_vs_atleast3}
For $n \ge 3$,
\[\Cexp{\Delta X^{\geq 2}_{n}}{ T_n}\geq \frac{1}{3n}X_n^{\geq 3}.\]
\end{lem}
\begin{proof}
Note that $\Delta X_n^{\ge 2}=1$ if and only if $W_n$ is a leaf. Therefore, as observed before, if $n \ge 3$ then

\begin{align}\label{eq:DeltaX2IsLarge}
    \Cexp{\Delta X^{\geq 2}_{n}}{  T_n}=\frac{1}{n}\sum_{\{v:D_n(v)\geq 2\}}\frac{L_{n}(v)}{D_{n}(v)}.
\end{align}
Let $T'_n$ be equal to $T_n$ with all of its leaves removed. We consider different sets of vertices in $T'_n$ and we study their contribution to the sum above. 

Let $V_1$ be the set of vertices of $T_n$ that are leaves in $T'_n$ and that were vertices of degree $2$ in $T_n$, so that for each $v\in V_1$, $\frac{L_{n}(v)}{D_{n}(v)}= \frac{1}{2}$. Let $V_2$ be the vertices of $T_n$ that are leaves in $T'_n$ that were vertices of degree at least $3$ in $T_n$ so that for each $v\in V_2$, $\frac{L_{n}(v)}{D_{n}(v)}=\frac{D_{n}(v)-1}{D_{n}(v)}\geq \frac{2}{3}$. Let $V_3$ be the vertices that have degree $2$ in $T'_n$ and had degree at least $3$ in $T_n$, so that for each $v\in V_3$, $\frac{L_{n}(v)}{D_{n}(v)}=\frac{D_n(v)-2}{D_n(v)}\geq \frac{1}{3}$. Finally, let $V_4$ be the vertices that have degree at least $3$ in $T'_n$. 
Therefore, 

\begin{equation}\label{eq:SumLOverDIsLarge}
    \sum_{\{v:D_n(v)\geq 2\}}\frac{L_{n}(v)}{D_{n}(v)}\geq \frac{1}{2}|V_1|+\frac{2}{3}|V_2|+\frac{1}{3}|V_3|~.
\end{equation}
To lower bound this sum note that $|V_1|+|V_2|$ is the number of leaves in $T'_n$. Since $T'_n$ is a tree, the number of leaves in $T'_n$ is given by 

$$\sum_{\{v: |\mathcal{N}(v,T'_n)|\geq 3\}} (|\mathcal{N}(v,T'_n)|-2)+2$$ and so 

$$|V_1|+|V_2|=\sum_{v\in V_4} (|\mathcal{N}(v,T'_n)|-2)+2~.$$
Finally, since $\sum_{v\in V_4} (|\mathcal{N}(v,T'_n)|-2)+2\geq |V_4|$ we obtain that 
$$ |V_1|+|V_2| \geq |V_4|~,$$
hence
$$ \frac{1}{2}|V_1|+\frac{2}{3}|V_2| \geq \frac{1}{3}|V_2|+\frac{1}{3}|V_4|~.$$
It also holds that $X_n^{\geq 3}=|V_2|+|V_3 |+| V_4|$, so we conclude that
\[ \frac{1}{2}|V_1|+\frac{2}{3}|V_2|+\frac{1}{3}|V_3| \geq \frac{1}{3}X_n^{\geq 3}~. \qedhere\]
Combined with \eqref{eq:DeltaX2IsLarge} and \eqref{eq:SumLOverDIsLarge}, this completes the proof. 
\end{proof}

\begin{lem}\label{lem:atleast3large_old}
Let $\alpha=(\sqrt{13}-3)/2 \approx 0.303$ be the unique positive solution of $x=\tfrac{1-2x}{1+x}$. Then, for any $\beta \in (0,\alpha)$ there exists $c>0$ such that 
\[\p{X_n^{\geq 3}<\beta X_n^{\geq 2}-c\log n \text{ infinitely often}} = 0,\]
and for all $n$, $\E{X_n^{\geq 3}}\geq \beta \E{X_n^{\geq 2}}-c\log n$.
\end{lem}
\begin{proof}
  The statements follow directly by applying Proposition \ref{prop:`lines'}, once we show that for $\alpha$ as in the Lemma statement, for $n\ge 4$, either $X^{\geq 3}_n\geq \alpha X^{\geq 2}_n$ or $\Cexp{\Delta X^{\geq 3}_{n}}{ T_n }\geq \alpha\Cexp{\Delta X^{\geq 2}_{n}}{ T_n }.$
 Suppose that $X^{\geq 3}_n< \alpha X^{\geq 2}_n$. Then, by \eqref{eq:diff1},

 \begin{equation}\label{eq:non-leavesgrowfast}\Cexp{\Delta X_{n}^{\geq 2} }{ T_n } \leq \frac{1}{2n}X_n^{\geq 2}+\frac{1}{2n}X_n^{\ge 3}\leq \frac{1+\alpha}{2n}X_n^{\geq 2}. \end{equation}
 Moreover, observe that the case $k=2$ of Lemma \ref{lem:fastgrowthverticesdegreek} gives that

\begin{align}
X^{2}_n&=X^{ 2,>2}_n+X^{2,\leq 2 }_n,\label{eq:diff6}\\
X^{\geq 3}_n&\geq X^{2,>2}_n\text{, and}\label{eq:diff7}\\
\Cexp{\Delta X^{\geq 3}_{n}}{ T_n}&\geq \frac{1}{2n}X^{2,\leq 2}_n.\label{eq:diff8}
\end{align}
Since $X^{\geq 3}_n< \alpha X^{\geq 2}_n$, \eqref{eq:diff7} implies that  $X^{2,>2}_n < \alpha X^{\geq 2}_n$. Note that $X^{\geq 2}_n=X^{2,>2}_n+X^{2,\leq 2}_n+X^{\geq 3}_n$, so the bounds $X^{2,>2}_n < \alpha X^{\geq 2}_n$ and $X^{\geq 3}_n < \alpha X^{\geq 2}_n$ together imply that $X^{2,\leq 2}_n>(1-2\alpha)X^{\geq 2}_n$. Combining this bound with \eqref{eq:diff8} and \eqref{eq:non-leavesgrowfast}, we conclude that 
\[ \Cexp{\Delta X^{\geq 3}_{n}}{ T_n}\geq \frac{1-2\alpha}{2n}X^{\geq 2}_n\geq \frac{1-2\alpha}{1+\alpha}\Cexp{\Delta X_{n}^{\geq 2}}{T_n }=\alpha \Cexp{\Delta X_{n}^{\geq 2} }{ T_n },
\]
as required.\qedhere

\end{proof}

\begin{lem}\label{lem:deg3toinfty}
    As $n\to \infty$, $X_n^{\geq 3}\to\infty$ almost surely.
\end{lem}

\begin{proof}
By Theorem \ref{thm:Diameter}, the diameter of $T_n$ is $\Theta(\log n)$ almost surely. Thus,
\begin{align}\label{eq:atleast2_as_inf}
    X_n^{\geq 2}\to \infty \text{ almost surely}.
\end{align}
Suppose for a contradiction that $X_n^{\geq3}$ does not go to infinity almost surely. That is, there exists a positive constant $c$ such that 
$$\p{ \forall n \in \N, \  X_n^{\geq3}<\infty  }= c >0.$$
By continuity of probability, this implies that there exists some constant $K$ such that
\begin{align}\label{eq:atleast3_finite_pos_prob}
    \p{ \forall n \in \N, \  X_n^{\geq3} \leq K }\geq \frac{c}{2}>0.
\end{align}
Using the fact that $X_n^{\geq 2}$ goes to infinity almost surely and $X_n^2 = X_n^{\geq 2} - X_n^{\geq 3}$ for all $n$, we have 

\begin{align}\label{eq:degree2_infty_cond_degree3}
     \Cprob{X_n^2 \rightarrow \infty}{\forall n \in \N, \  X_n^{\geq3} \leq K } = 1.
\end{align}
Define $\tau$ to be the smallest time after which, for all $n\geq \tau$, $T_n$ always contains at least two neighbouring vertices each of degree two,
$$\tau := \inf\{m \colon \forall n\geq m, \exists\, u,v \in T_n,  \, u \sim v,\, d_{n}(u) = d_{n}(v) = 2\}.$$
It follows from (\ref{eq:degree2_infty_cond_degree3}) that $\Cprob{\tau < \infty}{\forall n \in \N, \  X_n^{\geq3} \leq K } = 1.$ Note that $\Delta X_n^{\geq 3} = 1$ if and only if vertex $n+1$ attaches to a vertex of degree two. At time $n\geq \tau$ this occurs with probability at least $1/n$. Thus, except for finitely many $n$,
$$ \Cexp{\Delta X_n^{\geq 3}}{ T_n,\, X_n^{\geq 3}\leq K} \geq \frac{1}{n}~;$$
this implies that $\p{\lim_n X_n^{\geq 3}> K}=1$, which contradicts our hypothesis.
\end{proof}

\begin{lem}\label{lem:almostsureconvergencenonleaves} As $n\to \infty$, $X_n^{\geq 2}/\log n\to \infty$ almost surely.
\end{lem}
\begin{proof}
    Fix $C>0$. We will show that 
    \[ \liminf_{n\to \infty} \frac{X_n^{\geq 2}}{\log n}\geq C\text{ almost surely,}\]
    which implies the statement. Conditionally on $T_n$, $\Delta X_n^{\geq 2}$ is a Bernoulli random variable with parameter $\Cexp{\Delta X_n^{\geq 2}}{T_n}$. We prove that $\Cexp{\Delta X_n^{\geq 2}}{T_n}<C/n$ only finitely many times, in order to couple $(\Delta X^{\geq 2}_n)_{n\geq 1}$ to a sequence of independent $\operatorname{Bernoulli}$ random variables with parameter $C/n$. \\
    Let $(U_i,i \ge 1)$ be independent  uniform random variables on $[0,1]$.  Conditionally on $T_n$, construct $T_{n+1}$ from $T_n$ by setting $\Delta X_n^{\geq 2}=1$ if and only if $U_n<\Cexp{\Delta X_n^{\geq 2}}{T_n}$, and then sampling the additional randomness required to construct $T_{n+1}$ conditionally on $T_n$ and on the value of $\Delta X_n^{\geq 2}$. Define a coupling between  $(\Delta X^{\geq 2}_n)_{n\geq 1}$ and $(B_n)_{n\geq 1}$, a sequence of independent Bernoulli random variables with parameter $C/n$, by setting $B_n=1$ if  $U_n<C/n$ and $B_n=0$ otherwise. It is immediate that $\Delta X_n^{\geq 2}\ge B_n$ whenever $\Cexp{\Delta X_n^{\geq 2}}{T_n}\ge C/n$.\\
    Lemma \ref{lem:atleast2_vs_atleast3} states that $\Cexp{\Delta X_n^{\geq 2}}{ T_n}\geq \tfrac{1}{3n} X_n^{\geq 3}$, and by Lemma \ref{lem:deg3toinfty}, $X_n^{\ge 3}\to \infty$ almost surely, thus
    \begin{equation}\label{eq:largeexpectedincrementnonleaves}\Cexp{\Delta X_n^{\geq 2}}{ T_n}=\omega\left(1/n\right)\end{equation}
    almost surely. Therefore, almost surely, $\Cexp{\Delta X_n^{\geq 2}}{T_n}<C/n$ only finitely many times, and also $\Delta X_n^{\geq 2}<B_n$ only finitely many times. In particular 
    \[ \liminf_{n\to \infty} \frac{X_n^{\geq 2}}{\log n}\geq \liminf_{n\to \infty} \frac{\sum_{i=1}^n B_i}{\log n}\, \]
    almost surely. We claim that 
    \begin{equation}\label{eq:conv_Bernoullisum_to_C}\frac{\sum_{i=1}^n B_i}{\log n}\to C\text{ almost surely.}\end{equation}
    
    Indeed, let $(Y_j)_{j\geq 1}$ be a sequence of independent random variables satisfying $Y_j=\sum_{i=\lfloor e^{j-1}\rfloor +1}^{\lfloor e^{j}\rfloor}B_i$. Then $\lim_{j\to\infty}\e{[Y_j]}= C$ and $\E{Y_j^2}\leq 10C^2$. By Kolmogorov's strong law of large numbers \cite[Theorem~3.2.]{Durrett2016},
\[\frac{\sum_{j=1}^n Y_j}{n}\to C\text{ almost surely.}\]

This implies the convergence of \eqref{eq:conv_Bernoullisum_to_C} along the subsequence $(\lfloor e^j\rfloor)_{j\geq 1}$, and, by monotonicity,  \eqref{eq:conv_Bernoullisum_to_C}  follows.
  
\end{proof}
\begin{cor}\label{cor:x2_lbd}
    For $\alpha=(\sqrt{13}-3)/2$ the unique positive solution of $x=\tfrac{1-2x}{1+x}$, for any $\beta \in (0,\alpha)$ and for all $n$ sufficiently large, $\E{X_n^{\ge 3}} \ge \beta \E{X_n^{\ge 2}}$.
\end{cor}
    By Lemma~\ref{lem:atleast3large_old} we have $\e{X_n^{\geq 3}}\geq \beta \e{X_n^{\geq 2}}-c\log(n)$, and by Lemma~\ref{lem:almostsureconvergencenonleaves} $X_n^{\ge 2}=\omega(\log n)$ almost surely; Corollary~\ref{cor:x2_lbd} follows.

\begin{prop}\label{prop:urns}
    For $\alpha=(\sqrt{13}-3)/2$ the unique positive solution of $x=\tfrac{1-2x}{1+x}$, for any $0<\delta<\alpha/3 \approx 0.101$ we have $n^{-\delta}X_n^{\geq 2}\to \infty$ almost surely.
\end{prop}

\begin{proof}
By Lemmas \ref{lem:atleast2_vs_atleast3} and \ref{lem:atleast3large_old}, for all $\beta \in (0,\alpha)$ there exists a $c>0$ such that 

\begin{equation}\label{eq:Xn2_grows_as_polya}
    \p{\Cexp{\Delta X_n^{\geq 2}}{ T_n}< \frac{\beta}{3n}(X_n^{\geq 2 } -c\log n) \text{ infinitely often}} = 0~.
\end{equation}
Fix $\delta$ such that $0<\delta<\beta/3$ and fix $\gamma$ rational such that $\delta<\gamma<\beta/3$. Lemma \ref{lem:almostsureconvergencenonleaves} states that $X_n^{\geq 2 }=\omega(\log n)$ almost surely, therefore, almost surely there are only finitely many $n$ such that

\[\frac{\beta}{3}(X_n^{\geq 2 } -c\log n)<\gamma X_n^{\geq 2 },\]
Define the time $\tau$ as

\[\tau=3 \vee \sup\left\{n \geq 1:\Cexp{\Delta X_n^{\geq 2}}{T_n}< \frac{\gamma}{n}X_n^{\geq 2 }\right\}.\]
As a consequence of \eqref{eq:Xn2_grows_as_polya}, $\tau<\infty$ almost surely. For $n> \tau$, by the definition of $\tau$, the probability that $X_n^{\geq 2}$ increases at time $n$ is bounded from below by $\gamma X_n^{\geq 2}/n$. It is therefore natural to compare $X_n^{\geq 2}$ to a generalised P\'olya urn. We first introduce an urn process, containing $B_n$ black balls, and show that $n^{-\delta}B_n\to \infty$ almost surely. Conditionally on $\tau=t$, we then couple the sequences $(X^{\geq 2}_n)_{n\geq t}$ and $(B_n)_{n\geq t}$ and conclude the proof of Proposition~\ref{prop:urns}.

To introduce the urn process, let $M>0$ be an integer such that $\gamma^{-1}M$ is a positive integer, and let $t>0$ be another integer. Consider the urn process started at time $t$ with $M$ black balls and $t\gamma^{-1}M - M$ white balls. At every time step, draw a ball from the urn uniformly at random and return it to the urn together with $\gamma^{-1}M$ additional balls. If the drawn ball is white, all of the additional balls are white. If the drawn ball is black, $M$ of the additional balls are black and the other $(\gamma^{-1}-1)M$ balls are white. Denote by $B_n$ the number of black balls in the urn at time $n$. Then at time $n$, the number of black balls $B_n$ increases by $M$ with probability $B_n/(n\gamma^{-1}M)$. The described urn is \emph{triangular} since if a white ball is drawn only white balls are added to the urn. The asymptotic behaviour of triangular urn processes has been studied by \citet{Urnschemes}. Theorem 1.3.(v) in \cite{Urnschemes} implies that $n^{-\gamma}B_n$ converges almost surely to some random variable $Z$, and Theorem 8.7.\ in \cite{Urnschemes} shows that $Z$ puts no mass on $0$, so $n^{-\delta}B_n\to \infty$ almost surely as $\delta < \gamma$.

We now introduce a coupling satisfying that $(X_n^{\geq 2})_{n\geq t}$ grows at least as fast as $(B_n)_{n\geq t}$, on the event $\tau < t$. To formalise this coupling, note that conditionally on $\tau=t$, if $n\ge t$, then $\Cexp{\Delta X_n^{\geq 2}}{T_n} \geq \frac{\gamma}{n} X_n^{\geq 2}.$ Since $\Delta X_n^{\geq 2}\in\{0,1\}$,
$$\Cprob{\Delta(MX_n^{\geq 2})=M}{T_n}\geq \frac{MX_n^{\geq 2}}{n\gamma^{-1}M}~.$$
Remark that $MX_t^{\ge 2}$ is at least $M$. Let $(U_n)_{n\geq 1}$ be a sequence of independent uniform random variables on $[0,1]$. Conditionally on $T_n$, construct $T_{n+1}$ from $T_n$ by setting $\Delta X_n^{\geq 2} =1$ if and only if $U_n \leq \Cexp{\Delta X_n^{\geq 2}}{T_n}$ and sampling the remaining randomness in $T_{n+1}$ conditional on $T_n$ and the value of $\Delta X_n^{\ge 2}$. Let $B_t = 0$ and for $n\geq t$ let $\Delta B_n = M$ if $B_n/(n\gamma^{-1}M)\le U_n$ and $\Delta B_n = 0$ otherwise. Then $(B_n)_{n\geq t}$ is distributed as the number of black balls in the P\'olya urn described above. We already noted that $n^{-\delta}B_n\to \infty$ almost surely, and by our coupling we have that on the event $\tau<t$, $MX_n^{\ge 2}\geq B_n$ for all $n\geq t$. Thus, for $\varepsilon>0$

$$\p{n^{-\delta}X_n^{\ge 2} \to \infty}>\p{\tau<t}\ge 1-\varepsilon~.$$ 
Since $\tau$ is almost surely finite, $\varepsilon$ can be chosen arbitrarily small, by taking $t$ large, which concludes the proof.

\end{proof}

\begin{lem}\label{lem:atleast3small}
Let $\gamma=3-2\sqrt{2}\approx 0.172$ be the unique positive solution to $x=\tfrac{1-5x}{1-x}$. Then for any $\beta \in (1-\gamma,1)$ there exists $C>0$ such that \[\p{X_n^{\geq 3} > \beta X_n^{\geq 2} + C\log n \text{ infinitely often}} = 0\] and for all $n$, $\E{X_n^{\geq 3}} \leq \beta \E{X_n^{\geq 2}} + C\log n$.

\end{lem}

\begin{proof}
We apply Proposition~\ref{prop:`lines'} to the sequences $X_k = X_n^{\geq 2}$ and $Y_k = X_n^{\geq 3}$. Note that $\Delta X_n^{\geq 2} + \Delta X_n^{\geq 3}\in \{0,1\}$. It remains to show that there exists $\gamma>0$ such that for all $n$, either 
\begin{align*}
    X_n^{\ge 2}&\ge \frac{1}{1-\gamma} X_n^{\ge 3}\text{ or}\\
    \Cexp{\Delta X_n^{\ge 2}}{T_n}&\ge \frac{1}{1-\gamma}\Cexp{\Delta X_n^{\ge 3}}{ T_n}. 
\end{align*}
Lemma~\ref{lem:atleast3small} follows then directly by applying Proposition~\ref{prop:`lines'}. Suppose that $X_n^{\ge 2}< \frac{1}{1-\gamma} X_n^{\ge 3}$. Since $X_n^2=X_n^{\ge 2}-X_n^{\ge 3}$, \eqref{eq:diff2} gives that 
\[ \Cexp{\Delta X_n^{\ge 3}}{ T_n} \le \frac{4}{3n}(X_n^{\ge 2}-X_n^{\ge 3})<\frac{4\gamma}{3n}X_n^{\ge 2}. \]
From Lemma \ref{lem:atleast2_vs_atleast3} we see that \[\Cexp{\Delta X_n^{\geq 2}}{ T_n}\geq \tfrac{1}{3n}X_n^{\geq 3}>  \tfrac{1-\gamma}{3n}X_n^{\geq 2}.\] 
Therefore,
\[ \Cexp{\Delta X_n^{\geq 2} }{ T_n}>\frac{1-\gamma}{4\gamma}\Cexp{\Delta X_n^{\ge 3}}{T_n} . \]
The statement follows from the choice of $\gamma$.
\end{proof}

\begin{prop}\label{lem:upperbound_nonleaves}
Let $\delta=1-\gamma/2$, where $\gamma=3-2\sqrt{2}$. Then, for any $\delta<\lambda<1$, 
$$n^{-\lambda}X_n^{\geq 2}\to 0,$$
almost surely. 
\end{prop}
\begin{proof}
By \eqref{eq:diff1}, $\Cexp{\Delta X_{n}^{\geq 2} }{ T_n } \leq \frac{1}{2n}X_n^{\geq 2}+\frac{1}{2n}X_n^{\ge 3}$. Combining this with Lemma~\ref{lem:atleast3small} gives that for any $\beta\in (1-\gamma,1)$, there exists a $C>0$ such that 
    \[\p{\Cexp{\Delta X_{n}^{\geq 2} }{  T_n } >\frac{1+\beta}{2n}X_n^{\geq 2}+C \log n\text{ infinitely often}}=0.\]
By mimicking the proof of Proposition~\ref{prop:urns}, we can compare $X_n^{\geq 2}$ to a generalised P\'olya urn and obtain an upper bound on $\Delta X_{n}^{\geq 2}$. Omitting the details of the coupling, we conclude that for $\lambda \in (\frac{1+\beta}{2}, 1)$ it holds that  $\p{n^{-\lambda}X_n^{\geq 2} \to 0}=1$.
\end{proof}

\begin{proof}[Proof of Theorem~\ref{thm:SmallDegreeVertices} and \ref{thm:all_finite_deg_large}]
The second part of Theorem~\ref{thm:SmallDegreeVertices} follows from Proposition~\ref{lem:upperbound_nonleaves}, by noting that $X^{\geq k}_n\leq X^{\geq 2}_n$ and therefore, for any  $k\geq 2$, 

$$n^{-\lambda}X^{\geq k}_n\to 0 \text{ a.s.}$$
The upper bound in Theorem~\ref{thm:all_finite_deg_large} follows directly since $X_n^{\geq k+1} \leq X_n^{\geq k}$ for all $k$.  

By Proposition~\ref{prop:urns}, $\lim_n n^{-\delta}X_n^{\geq 2}= \infty$ almost surely. We prove the remaining cases in  the first part of Theorem \ref{thm:SmallDegreeVertices} and the lower bound of Theorem \ref{thm:all_finite_deg_large}, by using induction to prove that, almost surely, for all $k\geq 2$,

\begin{enumerate}
    \item[(i)] there exists a positive constant $c_k$ such that $\liminf_n X_n^{\geq k+1}/X_n^{\geq k}>c_k$, and
    \item[(ii)] $\lim_n n^{-\delta}X_n^{\geq k+1}=\infty$.
\end{enumerate}
The fact that $\lim_n n^{-\delta}X_n^{\geq 2}= \infty$ almost surely and Lemma~\ref{lem:atleast3large_old} imply that (i) holds for $k=2$ for any $0<c_2<\alpha$. This then also implies (ii) for $k=2$.

Now, fix $k\geq3$ and suppose that the induction hypothesis holds for all $2\leq \ell\leq k-1$. Let $b_k\in (0,1)$ be the solution to

\[ \frac{b_k}{1-2b_k}=\frac{c_{k-1}(k-1)}{k(k-3/2)} ,\]
and fix $0<c_k < b_k$. We claim that for all $n$, either

\begin{align*}X^{\geq k+1}_n&\geq b_k X^{\geq k}_n\text{ or} \\
 \Cexp{\Delta X^{\geq k+1}_{n}}{ T_n} & \geq b_k \Cexp{\Delta X^{\geq k}_{n}}{ T_n}\end{align*} 
 almost surely, except at finitely many times. If this holds, then applying Proposition \ref{prop:`lines'} gives us that there exists $C>0$ such that

 \[\p{X_n^{\geq k+1} < c_k X_n^{\geq k} - C\log n \text{ infinitely often}} = 0.\]
 By the induction hypothesis, $n^{-\delta}X_n^{\geq k}\to \infty$ and so
 
 $$\p{X_n^{\geq k+1}/{X_n^{\geq k}}< c_k \text{ infinitely often}} =0~.$$
This implies that (i) holds at step $k$, which in turn implies part (ii).

\noindent It remains to prove the claim. Suppose that \[X^{\geq k+1}_n < b_k X^{\geq k}_n~.\]
From \eqref{eq:diff4}, $b_k X^{\geq k}_n > X_n^{k,>k}$, which combined with \eqref{eq:diff3} implies that
\[X_n^k< b_k X^{\geq k}_n +X_n^{k,\leq k}~. \]
Now, using \eqref{eq:diff5} gives 
\[\Cexp{\Delta X^{\geq k+1}_{n}}{ T_n} > \frac{k-1}{kn} \left( X_n^k-b_k X^{\geq k}_n\right) = \frac{k-1}{kn}\left( (1-b_k)X_n^{\geq k}-X^{\geq k+1} \right)~,\]
where the equality holds since $X_n^k=X_n^{\geq k}-X_n^{\geq k+1}$. Our assumption $X^{\geq k+1}_n < b_k X^{\geq k}_n$ then gives

\[ \Cexp{\Delta X^{\geq k+1}_{n}}{ T_n} \geq \frac{k-1}{kn} (1-2b_k)X_n^{\geq k}~.\]

The induction hypothesis for $k-1$ implies that almost surely $X_n^{\geq k}>c_{k-1}X_n^{\geq k-1}\ge c_{k-1} X_n^{k-1}$, except at finitely many times. Therefore,
\[ \Cexp{\Delta X^{\geq k+1}_{n}}{ T_n} \geq \frac{c_{k-1}(k-1)}{kn}  (1-2b_k)X_n^{k-1}~\]
except at finitely many times. But \eqref{eq:upperbound_change} implies that for all $n$ sufficiently large we have that $ X_n^{k-1}\ge \tfrac{k-3/2}{n}\Cexp{\Delta X_n^{\ge k}}{T_n}$, so we conclude that, except at finitely many times

\[ \Cexp{\Delta X^{\geq k+1}_{n}}{ T_n} \geq \frac{c_{k-1}(k-1)}{k(k-3/2)} (1-2b_k)\Cexp{\Delta X^{\geq k}_{n}}{ T_n}~. \]
This proves the claim by our choice of $b_k$, which concludes the proof. 
\end{proof}

\section{Open questions and future directions}\label{sec:questions}
We conclude with some open questions about the random friend tree.
\begin{enumerate}
    \item In Theorem \ref{thm:SmallDegreeVertices} we prove that, for some $0.1< \delta <\lambda< 0.9$, almost surely $n^{\delta} \ll X_n^{\geq 2} \ll n^{\lambda}$. A question of interest would be whether the gap between the upper and lower bound can be closed and whether, for some $\mu>0$ and some random variable $X$ with non-trivial support, $n^{-\mu}X_n^{\geq 2} \to X$ almost surely. Simulations by \citet{MR3683821} suggest that $X_n^{\geq 2}$ grows as $n^{\mu}$, with $\mu \approx 0.566$. 
    \item We prove that, for fixed $k$, the number of vertices with degree at least $k+1$ is of the same order as the number of vertices with degree at least $k$, see Theorem~\ref{thm:all_finite_deg_large}.  Can we prove, for fixed $k$, that the number of degree-$k$ nodes is of the same order as the the number of degree-$(k+1)$ nodes? Does it hold that $\limsup_{n\to\infty}\frac{X_n^{\geq k}}{X_n^{\geq 2}}$ goes to $0$ as $k$ goes to infinity? Or, informally, are most of the non-leaves vertices of bounded degree?
    \citet{MR3683821} conjecture that for each $k$, $\smash{\frac{X_n^{k}}{X_n^{\geq 2}}}$ has an almost sure limit  that is $\Theta(k^{-(1+\mu)})$.
    \item   Is $\p{Z_u>0}$ decreasing in $u$? More generally, does $Z_u$ stochastically dominate $Z_v$ for $u<v$?
    \item Does it hold that $\sum_{i\ge 1}Z_i=1$ almost surely?
    \item We know that every edge contains a vertex of linear degree, but the diameter of the tree grows logarithmically, so there must be connected subtrees consisting of low-degree vertices whose linear growth has not kicked in yet. This is illustrated by the proof of  Proposition \ref{prop:leafdepth}, which shows that there are paths of length $\Theta(\log(n)/\log\log(n))$ that  consist of just degree $2$ vertices. It would be interesting to get a better understanding of the law of these exceptional substructures that contain most of the low-degree vertices. What does the forest induced by the vertices of degree at most $N$, for large $N$, look like?  Do these subtrees look like `young' friend trees?
    \item A natural extension of the model is to attach the new vertex to multiple, say $m$, vertices. There are two variants: either $V_n$ is a uniformly random vertex and the new vertex $n+1$ attaches to $m$ independently sampled random neighbours of $V_n$, or we let $V_n^{(1)},\dots, V_n^{(m)}$ to be independent random vertices and we let $n+1$ connect to a uniform neighbour of each of the $V_n^{(i)}$. 
    \item A second variation is to  choose $0<p<1$ and connect to $V_n$ with probability $p$ and to $W_n$ with probability $1-p$. This modification makes it much easier for neighbours of high-degree vertices to grow their degree, and in particular, the degree of every vertex goes to infinity almost surely as the tree grows. It would be interesting to see how much of the structure of the random friend tree remains after this modification.
    \item Another final modification of the model, as described in the introduction, is to let $W_n$ be the endpoint of a random walk with $k$ steps rather than $1$ step from $V_n$. In the case of $k =0$, we obtain an URRT and if $k$ is sufficiently large such that the random walk is perfectly mixed, we get a PA tree. One could study how properties such as the size of the largest degree depend on $k$.
\end{enumerate}

\addtocontents{toc}{\SkipTocEntry} 

\section{Appendix}
We prove Proposition~\ref{prop:`lines'}. We start by stating and proving a technical lemma that is needed for its proof. 

We make use of the following straightforward fact. Fix $a,b > 0$ and let $(Y_k)_{k \ge 0}$ be a random walk with steps in $\{-a,b\}$ such that $\E{\Delta Y_k}=c>0$. Then, $\p{\Delta Y_k=b}=(a+c)/(a+b)$ 
and by writing $\tau =\inf\{k \ge 0: Y_k < 0\}$, we have $\p{\tau =\infty}>0$.

\begin{lem}\label{lem:StickyCoupling}
Let $B=(B_k)_{k\geq 0}$ be a random process adapted to a filtration $(\cF_k)_{k\geq 0}$. Suppose that there exist $a,b>0$ such that, almost surely for each $k$, $\Delta B_{k}\in\{-a,0,b\}$. Suppose further that there exists a constant $c>0$ such that 
\[\E{\Delta B_k \mid \cF_k , B_k < 0, \Delta B_k \neq 0 }\ge c~. \]
Then, there exists a constant $C=C(a,b,c)$ such that
\[ \p{B_n<-C\log n\text{ infinitely often}}=0,\]
and $\E{ B_n} \ge -C\log n$.
\end{lem}
\begin{proof}
We bound $B$ from below by another, simpler process $S=(S_n)_{n \ge 0}$. 
The conditions of the lemma imply that we may couple $B$  with a sequence $(Y_k)_{k \ge 0}$ of independent, random variables taking values in $\{-a,b\}$ with 
\[\p{Y_k=b}=\frac{a+c}{a+b}=1-\p{Y_k=-a},\]
such that for all $k \ge 0$, on the event that $B_k < 0$ and $\Delta B_k \ne 0$ we have $\Delta B_k \ge Y_k$. We define $S$ via the following coupling with $B$:
\begin{enumerate} 
\item if $B_n \ge0$, then $S_{n+1}=-a$,
\item if $B_n<0$ and $\Delta B_n=0$ then $\Delta S_n=0$,
\item if $B_n<0$ and $\Delta B_n\ne 0$ then $\Delta S_n=Y_n$.
\end{enumerate}
An illustration of the coupling can be found in Figure~\ref{fig:sticky_coupling}. Since $B_n \ge 0$ implies $B_{n+1}\ge -a$, it is immediate that $S_n \le B_n$ for all $n\ge 0$. The process $S$ is a sequence of independent negative (incomplete) excursions of a random 
process that, restricted to the non-constant steps, has independent increments with positive mean except for finitely many times almost surely. There are at most $n/2$ such excursions by time $n$, and if we collapse the constant steps, they are independent realisations of a random walk with step size in $\{-a,b\}$ and drift $c$, started at $-a$ and ended before (or when) reaching $0$. To understand their minimum, let us denote by $R=(R_n)_{n \ge 0}$ a random walk starting at $0$, with steps in $\{-a,b\}$ and strictly positive drift $c$. We define
$$\tau_1=\inf\left\{ k \colon \ R_k< 0 \right\},$$
and
$$\tau_{\ell}=\inf\left\{ k> \tau_{\ell-1}\colon R_k<R_{\tau_{\ell-1}} \right\}.$$
With positive probability, $R$ stays positive forever, and in particular, using the fact stated just before Lemma~\ref{lem:StickyCoupling}, there exists $0<p<1$ such that
$$
0<1-p=\p{\tau_{\ell}=\infty|\tau_{\ell-1}<\infty}.$$ 
Since the increments are bounded from below by $-a$, we know that $R_{\tau_{\ell}}-R_{\tau_{(\ell-1)}}\geq -a$, which together with the previous identity implies that
\[
\min_{k\geq1}\left\{R_k\right\} {\succeq}_{st} -a\cdot \operatorname{Geom}(1-p)~.
\]
From the definition of the coupling, we know that $B_n\geq S_n$ and that $S_n+a$ stochastically dominates the minimum of  $n$ realisations of $(R_k)_{k\geq 0}$. Let $(A_n)_{n \ge 1}$ be independent $\geom(1-p)$ random variables. Then
\[
\p{S_n \le -a(k+1)} \le \p{\max_{i \in [n]} A_i \ge k} \le np^k.
\]
The upper bound is at most $n^{-2}$ if $k \ge -3\log n/\log p$
and the first assertion follows from the Borel--Cantelli lemma. Taking $k=-3\log n/\log p+ \ell$, the above bound likewise implies that 
\[
\p{-\frac{S_n}{a} +3\log n/\log p\ge \ell +1 } \le n p^{-3\log n/\log p}p^\ell,
\]
and summing over $\ell \ge 0$ gives the bound
\[
\E{-\frac{S_n}{a} +3\log n/\log p} \le n p^{-3\log n/\log p}\frac{1}{1-p}\, ,
\]
which establishes the second assertion of the lemma.\qedhere
\begin{figure}\label{fig:sticky_coupling}
\begin{center}
\includegraphics[scale=0.7]{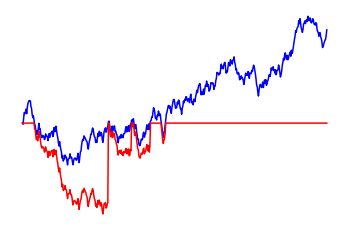}
\end{center}
\caption{Illustration of a coupling between $B$ (in blue) and $S$ (in red).}
\end{figure}
\end{proof}
We are now ready to prove the proposition. 
\begin{proof}[Proof of Proposition~\ref{prop:`lines'}] We start by proving Proposition~\ref{prop:`lines'} in the case where the hypothesis holds for all $n$ (and not all but finitely many $n$). Let $\beta \in (0,\alpha)$. We prove the proposition by applying Lemma \ref{lem:StickyCoupling} to the process $B = (B_k)_{k\geq 0}$ with $B_k:=X_k-\beta Y_k$. Note that $B$ has increments in $\{-\beta, 0,1\}$. We need to show that there exists a constant $c$ such that for all $k$, 
\begin{equation}\label{eq:edb_lowerbd}
\Cexp{\Delta B_k}{ \cF_k, B_k < 0, \Delta B_k \neq 0 }\geq c~.
\end{equation}
Define \[p^+_k=\Cprob{\Delta X_k=1 }{ \cF_k, (\Delta X_k, \Delta Y_k)\neq (0,0)}=\p{\Delta B_k=1\mid \cF_k, \Delta B_k \neq 0}\] and so 
\[1-p^+_k=\Cprob{\Delta Y_k=1 }{ \cF_k, (\Delta X_k, \Delta Y_k)\neq (0,0)}=\Cprob{\Delta B_k=-\beta }{\cF_k, \Delta B_k \neq 0}.\]
 Note that, if $\{B_k<0\}$, then $\{X_k<\beta Y_k\}$, and so $\Cexp{\Delta X_k }{ \cF_k}\geq  \alpha \Cexp{\Delta Y_k}{\cF_k}$, which implies that $p^+_k\geq \alpha (1-p^+_k)$. We split the event $\{B_k<0\}$ into two cases and show that $\Cexp{\Delta B_k}{\cF_k, \Delta B_k \neq 0 }$ is bounded below by some positive constant in both cases. If $B_k<0$ and $p^+_k\leq \frac{1/2+\beta}{1+\beta}$, then
\begin{align*}
\Cexp{\Delta B_k}{\cF_k, \Delta B_k \neq 0 }=p^+_k-\beta (1-p^+_k)
\geq (\alpha-\beta) (1-p^+_k)
\geq \frac{(\alpha-\beta) }{2+2\beta}.
\end{align*}
On the other hand, if $B_k<0$ and $p^+_k> \frac{1/2+\beta}{1+\beta}$, then 
\begin{align*}
\Cexp{\Delta B_k}{\cF_k, \Delta B_k \neq 0 }=p^+_k(1+\beta)-\beta\geq 1/2.
\end{align*} 
Therefore, \eqref{eq:edb_lowerbd} holds, with a lower bound of $\min\left\{\frac{(\alpha-\beta) }{2+2\beta},1/2\right\}$ for $c$, and the claim then directly follows from Lemma~\ref{lem:StickyCoupling}.

Finally, we prove that the first statement in the proposition still holds when the assumptions fail at a finite number of times.

We call a time $k$ \emph{bad} when $\E{\Delta X_k \mid \cF_k}<  \alpha \E{\Delta Y_k \mid \cF_k}$ and $\{X_k < \alpha Y_k\}$; otherwise we call it \emph{good}. We couple $(X, Y)$ to a slightly modified process $(X',Y')$ that has the same increments as $(X,Y)$ except at bad times, and that satisfies the assumptions at all times. Observe that for each $k$, given $\cF_k$, we know whether $k$ is bad or not. If $k$ is bad, set $(\Delta X'_k,\Delta Y'_k)=(1,0)$. If $k$ is good, set $(\Delta X'_k,\Delta Y'_k)=(\Delta X_k, \Delta Y_k)$. We claim that $(X',Y')$ satisfies the assumptions at all times. The requirement $0\leq \Delta X'_k + \Delta Y'_k \leq 1$ for all $k$ is obviously satisfied. Moreover, for bad $k$, $\E{\Delta X'_k \mid \cF_k}=1$ and $\E{\Delta Y'_k \mid \cF_k}=0$, so at bad times the second requirement is also satisfied. Finally, by construction, $X'_k\ge X_k$ and $Y'_k\le Y_k$ for all $k$, so if $k$ is good and $\{X'_k < \alpha Y'_k\}$, then also $\{X_k < \alpha Y_k\}$, and therefore 
    \[\E{\Delta X'_k \mid \cF_k}=\E{\Delta X_k \mid \cF_k}\geq  \alpha \E{\Delta Y_k \mid \cF_k} = \alpha \E{\Delta Y'_k \mid \cF_k}.\]
    Then, the first part of the proof implies that for any $\beta\in (0,\alpha)$ there exists $C>0$ such that $\p{X'_n < \beta Y'_n - C\log n \text{ i.o.}} = 0$. Finally, for $B$ the total number of bad times, for each $k$ it holds that $X'_k\le X_k+B$ and $Y'_k\ge Y_k-B$. This implies that if $X_k< \beta Y_k - 2C\log k$ then either $X'_k < \beta Y'_k - C\log n$ or $2B> C\log k$. Therefore, 
    \[\p{X_n < \beta Y_n - 2C\log n \text{ i.o.}} \le \p{X'_n < \beta Y'_n -C\log n  \text{ i.o.}}+\p{B=\infty}=0.\qedhere\]
\end{proof}

\section*{Acknowledgements}

Louigi Addario-Berry is supported by the Natural Sciences and Engineering Research Council of Canada (NSERC) and the Canada Research Chairs program. 

Simon Briend is supported by Région Île de France.

Luc Devroye is supported by the Natural Sciences and Engineering Research Council of Canada (NSERC). 

Serte Donderwinkel is supported by the CogniGron research center and the Ubbo Emmius Funds (Univ. of Groningen). 

Céline Kerriou is supported by the DFG project 444092244 ``Condensation in random geometric graphs" within the priority programme SPP~2265. 

G\'abor Lugosi is supported by Ayudas Fundación BBVA a
Proyectos de Investigación Científica 2021 and
the
Spanish Ministry of Economy and Competitiveness grant PID2022-138268NB-I00, financed by MCIN/AEI/10.13039/501100011033,
FSE+MTM2015-67304-P, and FEDER, EU.


\small 

\bibliographystyle{plainnat}
\bibliography{template}

%
%
                 
\appendix

\end{document}